\documentclass{amsart}
\date{\today}
\usepackage{graphicx, color, enumerate}
\usepackage{amsmath, amssymb, amscd, mathrsfs}
\title[Isometric immersions with singularities]{%
  Isometric immersions with singularities
  between space forms of the same positive curvature
}
\author[A. Honda]{Atsufumi Honda}
\address{%
   National Institute of Technology, Miyakonojo College, 
   Yoshio, Miyakonojo 885-8567, Japan
}
\email{atsufumi@cc.miyakonojo-nct.ac.jp}
\usepackage{amsthm}
\theoremstyle{plain}
 \newtheorem{theorem}{Theorem}[section]
 \newtheorem{introtheorem}{Theorem}

 \newtheorem{proposition}[theorem]{Proposition}
 \newtheorem{fact}[theorem]{Fact}
 \newtheorem*{fact*}{Fact}
 \newtheorem{lemma}[theorem]{Lemma}
 \newtheorem{corollary}[theorem]{Corollary}
 \theoremstyle{remark}
 \newtheorem{definition}[theorem]{Definition}
 \newtheorem{remark}[theorem]{Remark}
 \newtheorem*{acknowledgements}{Acknowledgements}
 \newtheorem{example}[theorem]{Example}
\numberwithin{equation}{section}

\newcommand{\Z}{\boldsymbol{Z}}

\newcommand{\R}{\boldsymbol{R}}

\newcommand{\SO}{\operatorname{SO}}

\newcommand{\inner}[2]{\left\langle{#1},{#2}\right\rangle}
\newcommand{\vect}[1]{\boldsymbol{#1}}

\begin{document}
\begin{abstract}
  In this paper, we give a definition of 
  {\it coherent tangent bundles of space form type\/},
  which is a generalized notion of space forms.
  Then, we classify their realizations in the sphere as a wave front,
  which is a generalization of a theorem of O'Neill and Stiel:
  any isometric immersion of the $n$-sphere into the $(n+1)$-sphere 
  of the same sectional curvature is totally geodesic.
\end{abstract}
\maketitle


\section*{Introduction}
\label{intro}
Consider an isometric immersion of $\Sigma^n(c)$ into $\Sigma^{n+1}(c)$,
where $\Sigma^n(c)$ denotes a connected, simply connected 
$n$-dimensional space form of constant sectional curvature $c$.
Such an immersion must be 
\begin{itemize}
\item
a cylinder over a complete plane curve if $c=0$,
which was proved by Hartman-Nirenberg \cite{HN}.
(Massey \cite{Massey} gave an alternative proof for $n=2$.)
\item
totally geodesic if $c>0$,
which was proved by O'Neill-Stiel \cite{OS}.
\end{itemize}
Roughly speaking, if $c$ is non-negative, such an isometric immersion is trivial.
However, if $c$ is negative, 
there exist many nontrivial examples of such isometric immersions 
as shown in \cite{Nomizu,Ferus,AbeHaas}.
A complete classification was given by Abe-Mori-Takahashi \cite{AbeMoriTakahashi} 
(see also \cite{Honda}).

On the other hand, even if $c$ is non-negative,
there are many nontrivial examples of such isometric immersions
admitting some {\it singularities}
(cf.\ Figure \ref{fig:gallery}).
In the seminal paper of Murata-Umehara \cite{MU},
they classified complete {\it flat fronts},
which are regarded as a generalization of 
isometric immersions of $\R^2=\Sigma^2(0)$ into $\R^3=\Sigma^{3}(0)$
(cf.\ Section \ref{sec:transformation}).
Moreover, complete flat fronts were proved 
to have various global properties
(see Fact \ref{fact:MU} for the precise statement).
Hence, the following problem naturally arises:
{\it What occurs in the non-flat or higher dimensional cases?}

\begin{figure}[htb]
\begin{center}
 \begin{tabular}{{c@{\hspace{8mm}}c@{\hspace{6mm}}c}}
  \resizebox{2.8cm}{!}{\includegraphics{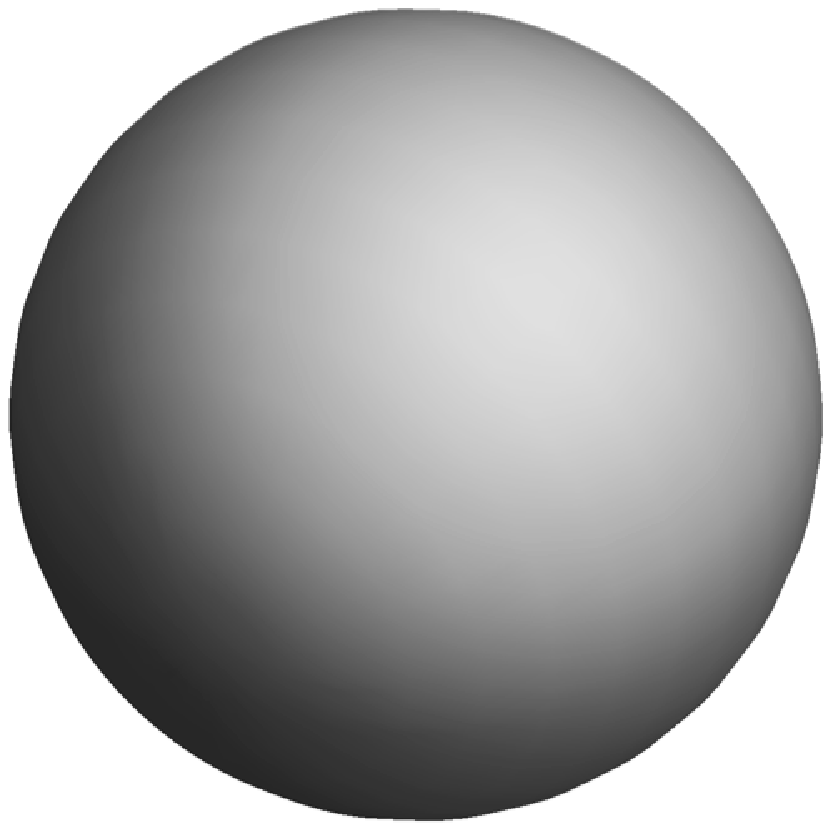}}&
  \resizebox{3cm}{!}{\includegraphics{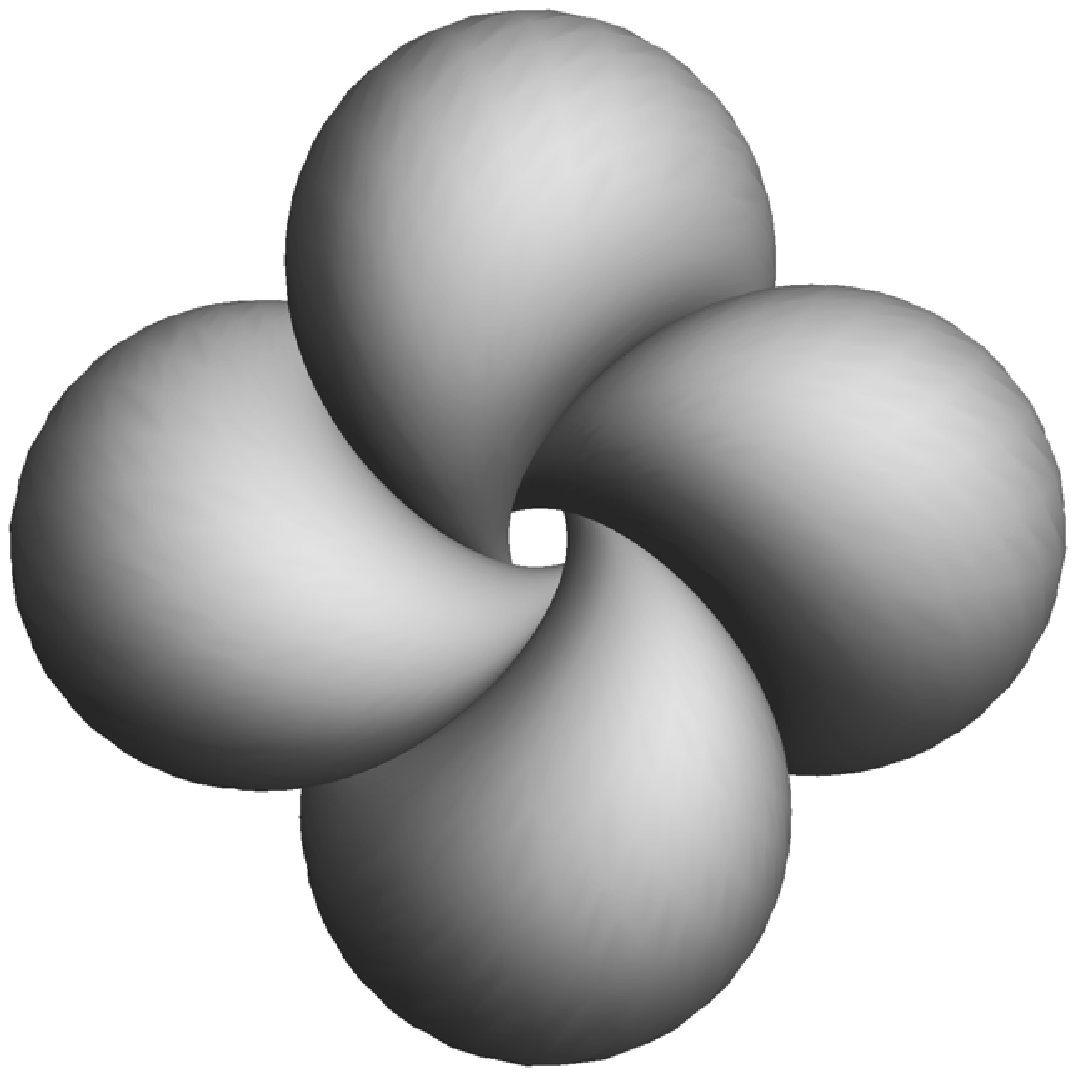}}&
  \resizebox{3cm}{!}{\includegraphics{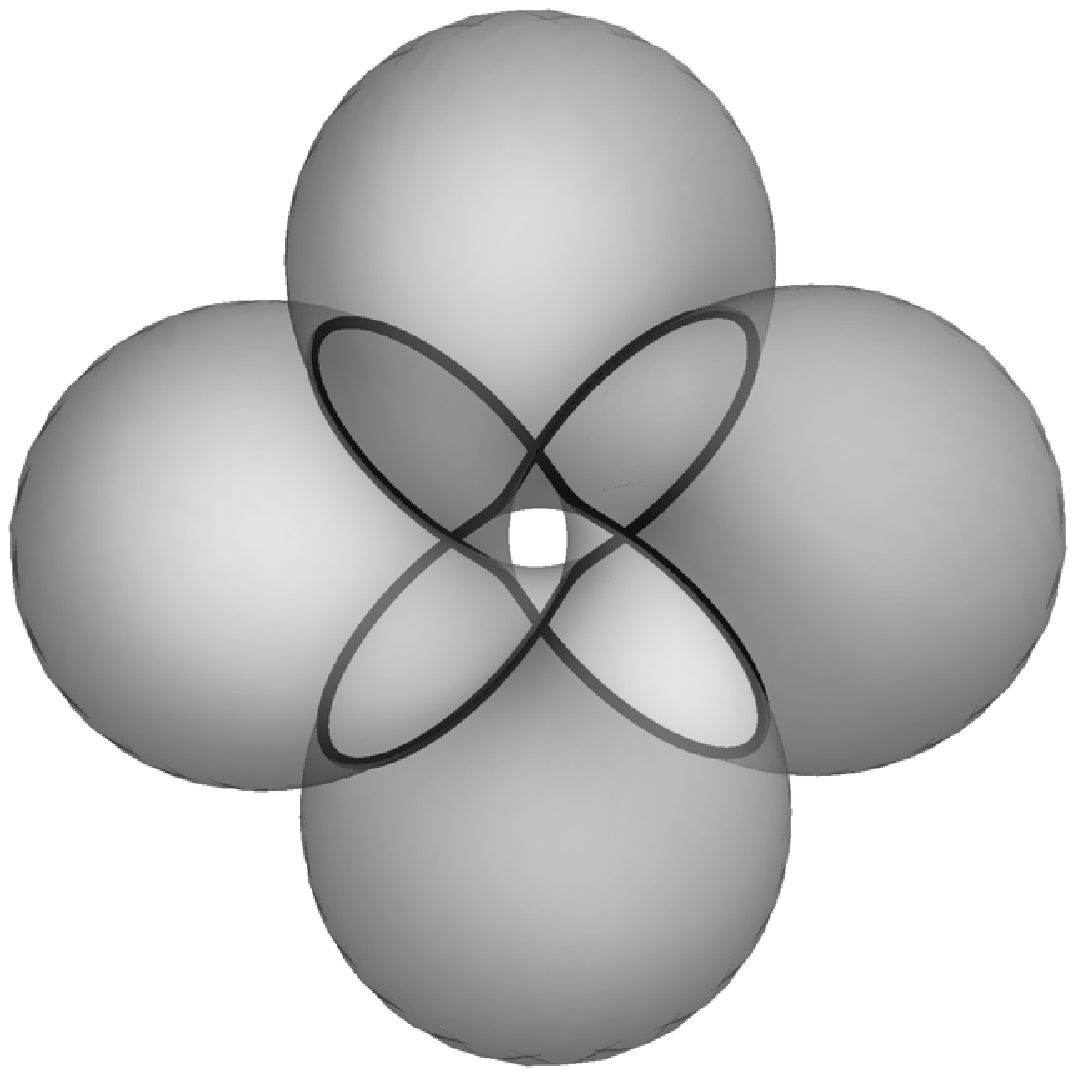}} \\
  {\footnotesize  (A) totally geodesic.} &
  {\footnotesize  (B) non-totally-geodesic.} &
  {\footnotesize  (B') the same, transparent.}
 \end{tabular}
 \caption{Images of isometric immersions of $S^2$ into $S^3$.
 Non-totally-geodesic ones have singularities
 (see Example \ref{ex:non-co-ori}).
 Throughout this paper, 
 we show graphics using the stereographic projection
 $\pi_{\rm st} : S^3\setminus\{(0,-1,0,0)\} \rightarrow \R^3$
 defined by
  $
    \pi_{\rm st}(x_1,x_2,x_3,x_4)
    =(x_1,x_3,x_4)/(1+x_2).
  $
 The bold curves are the image of the singular set.
 }
 \label{fig:gallery}
\end{center}
\end{figure}

In this paper, we investigate `generalized' isometric immersions
of $S^n(c)=\Sigma^n(c)$ into $S^{n+1}(c)=\Sigma^{n+1}(c)$,
where $S^{n+1}(c)=\Sigma^{n+1}(c)$ is a sphere
of positive sectional curvature $c>0$.
To clarify the settings, 
we first give a definition of {\it coherent tangent bundles of space form type},
which is a generalization of space forms (Definition \ref{def:CTB_spaceform}).
Recall that, in \cite{SUY0, SUY1}, Saji-Umehara-Yamada
introduced {\it coherent tangent bundles} 
which can be regarded as a generalized notion of 
Riemannian manifolds with singularities
(for the precise definitions, see Section \ref{sec:prelim}).
We remark that
there exist coherent tangent bundle structures of space form type on $S^{n}$
whose sectional curvature is non-positive (cf.\ Example \ref{ex:simp_conn_CTB}),
while the sectional curvature
of a simply connected compact space form must be positive.

A {\it wave front} (or a {\it front\/}, for short)
is a generalized notion of immersions admitting singularities.
Coherent tangent bundles can be considered as 
an intrinsic formulation of fronts, 
because fronts (more generally {\it frontals\/}) 
induce coherent tangent bundle structures on the source manifold
(cf.\ \cite{SUY1}, see also Example \ref{ex:induced_CTB}).

\

The main theorem of this paper is as follows.

\begin{introtheorem}\label{thm:main}
Let $M^n$ be a connected smooth $n$-manifold,
and $f : M^n \rightarrow S^{n+1}(c)$ a front.
If the induced coherent tangent bundle on $M^n$
is of space form type with positive sectional curvature $c$, 
then $f$ is either totally geodesic or 
an umbilic-free tube of radius $\pi/(2\sqrt{c})$ whose center curve is 
a closed regular curve in $S^{n+1}(c)$.
In particular, $M^n$ is diffeomorphic to 
either $S^n$ or $S^1\times S^{n-1}$,
hence orientable.
\end{introtheorem}

To prove Theorem \ref{thm:main},
we classify {\it weakly complete} 
constant curvature $1$ fronts
in $S^{n+1}$ (Theorem \ref{thm:w-complete}),
where completeness implies weak completeness 
(cf.\ Fact \ref{fact:MU_complete}).
Using these, we then investigate some local and global properties 
of the caustics of constant curvature $1$ fronts.
Caustics are defined as the singular loci of parallel fronts.
It is known that caustics of 
flat fronts in $\Sigma^3(c)$ are also flat fronts
(cf.\
\cite{Roitman} for the case of $H^3$; 
\cite{MU} for the case of $\R^3$; 
\cite{KitaUme} for the case of $S^3$).
We can prove a similar result for constant curvature $1$ fronts in $S^{n+1}$
(Proposition \ref{prop:caus-tube}).
Even if the original front is connected or co-orientable,
its caustic will, in general, not be connected or co-orientable.
However, the following holds.

\begin{introtheorem}\label{thm:caustic}
Let $f:M^n\rightarrow S^{n+1}$ be 
a connected non-totally-geodesic 
weakly complete constant curvature $1$ front.
Then, the caustic $f^C$ of $f$ is also 
a connected non-totally-geodesic
weakly complete constant curvature $1$ front.
Moreover, if $f$ is complete $($resp.\ non-co-orientable$)$ 
then $f^C$ is also complete $($resp.\ non-co-orientable$)$.
\end{introtheorem}

Finally in Section \ref{sec:transformation}, 
we investigate constant curvature $1$ fronts in the case of $n=2$.
On the regular point set of a constant curvature $1$ front,
its extrinsic curvature vanishes identically.
Thus, we also call such fronts {\it extrinsically flat}.
Although the local nature of extrinsically flat surfaces in $S^3$
is the same as that of flat surfaces in $\R^3$,
via central projection of a hemisphere to a tangent space
(cf.\ Example \ref{ex:dev-Moebius}),
they may display different global properties
(cf.\ Example \ref{ex:non-co-ori}).
On the other hand, in the particular case of $n=2$,
we can define the {\it duals} of tubes of radius $\pi/2$ in $S^3$.
Such duals may not be fronts in general (cf.\ Remark \ref{rem:INS}).
We classify tubes of radius $\pi/2$ in $S^3$ 
whose duals are congruent to the original ones
(i.e., {\it self-dual} developable tubes)
in Theorem \ref{prop:self-dual}.

This paper is organized as follows.
In Section \ref{sec:prelim}, 
we review the definition and fundamental properties 
of coherent tangent bundles and fronts.
In Section \ref{sec:CTB}, 
we give a definition of coherent tangent bundles of space form type
and give some examples.
In Section \ref{sec:proof_A}, 
we shall investigate constant curvature $1$ fronts in $S^{n+1}$,
and prove Theorem \ref{thm:main} and Theorem \ref{thm:caustic}.
In Section \ref{sec:transformation}, 
we restrict to the case of $n=2$ and investigate their duals.

\section{Preliminaries}
\label{sec:prelim}

In this section, we review the definitions and basic properties of 
space forms, coherent tangent bundles and (wave) fronts.

\subsection{Space form}
Let $S^{n+1}$ be the unit sphere 
$
  S^{n+1} = \{ \vect{x} \in \R^{n+2} \,;\, \inner{\vect{x}}{\vect{x}}=1 \},
$
where $\inner{~}{~}$ is the canonical inner product 
of the Euclidean space $\R^{n+2}$.
We denote by $H^{n+1}$ the hyperbolic space 
\[
  H^{n+1}=\left\{
    \vect{x}=(x_1,\cdots,x_{n+2})\in \R^{n+2}_1\,;\,
    \inner{\vect{x}}{\vect{x}}_L=-1,\, x_{1}>0
  \right\},
\]
where $\R^{n+2}_1$ is the Lorentz-Minkowski space
equipped with the canonical Lorentzian inner product
$\inner{\vect{x}}{\vect{x}}_L=-(x_1)^2+(x_2)^2+\cdots +(x_{n+2})^2$.
That is, $\Sigma^{n+1}(0)=\R^{n+1}$,
$\Sigma^{n+1}(1)=S^{n+1}$
and $\Sigma^{n+1}(-1)=H^{n+1}$,
where $\Sigma^n(c)$ is the connected, simply connected 
$n$-dimensional space form of constant sectional curvature $c$.
Then, the unit tangent bundles of $\R^{n+1}$, $S^{n+1}$ and $H^{n+1}$
are identified with
\begin{equation}\label{eq:UTbund}
\begin{split}
  &T_1\R^{n+1}=\R^{n+1}\times S^n,\\
  &T_1S^{n+1}= \{ (p,v) \in S^{n+1}\times S^{n+1} \,;\, \inner{p}{v}=0 \},\\
  &T_1H^{n+1}= \{ (p,v) \in H^{n+1}\times S^{n+1}_1 \,;\, \inner{p}{v}_L=0 \},
\end{split}
\end{equation}
respectively, where 
$
  S^{n+1}_1=\left\{ \vect{x}\in \R^{n+2}_1\,;\, \inner{\vect{x}}{\vect{x}}_L=1 \right\}
$
is the de Sitter space.
For $\vect{x}^1,\cdots, \vect{x}^{n+1}\in \R^{n+2}$,
set a vector 
$\vect{x}^1\wedge \vect{x}^2\wedge \cdots \wedge \vect{x}^{n+1}\in \R^{n+2}$
as
\begin{equation}\label{eq:wedge_prod}
  \vect{x}^1 \wedge \vect{x}^2 \wedge \cdots \wedge \vect{x}^{n+1}
  =\left|
    \begin{array}{cccc}
    e_1 & e_2 & \cdots & e_{n+2} \\
    x^1_1 & x^1_2 & \cdots & x^1_{n+2} \\
    \cdots & \cdots & \cdots & \cdots \\
    x^{n+1}_1 & x^{n+1}_2 & \cdots & x^{n+1}_{n+2} 
    \end{array}
  \right|,
\end{equation}
where 
$\{ e_1,\cdots,e_{n+2} \}$ is the canonical basis of $\R^{n+2}$
and $\vect{x}^j=(x^j_1 ,\, \cdots ,\, x^j_{n+2})$
for each $j=1,\cdots,(n+1)$.

\subsection{Coherent tangent bundle}

In \cite{SUY0, SUY1}, Saji-Umehara-Yamada 
gave a definition of {\it coherent tangent bundles\/},
which is a generalized notion of Riemannian manifolds.

\begin{definition}[\cite{SUY0, SUY1}]\label{def:CTB}
Let $\mathcal{E}$ be a vector bundle of rank $n$
over a smooth $n$-manifold $M^n$.
We equip a fiber metric $\inner{~}{~}$ on $\mathcal{E}$
and a metric connection $D$ on $(\mathcal{E},\inner{~}{~})$.
If $\varphi : TM^n \rightarrow \mathcal{E}$
is a bundle homomorphism which satisfies
\begin{equation}\label{eq:Codazzi}
  D_{X}\varphi(Y)-D_{Y}\varphi(X)-\varphi([X,Y])=0
\end{equation}
for smooth vector fields $X$, $Y$ on $M^n$,
then $\mathcal{E}=(M^n,\,\mathcal{E},\,\inner{~}{~},\,D,\,\varphi)$
is called a {\it coherent tangent bundle} over $M^n$.
\end{definition}

It is known that, in general, coherent tangent bundles can be constructed
only from positive semi-definite metrics (so-called {\it Kossowski metrics}).
This was proved in \cite{HHNSUY} for the case of $n=2$,
and in \cite{SUY3} for the higher dimensional case.

A point $p\in M^n$ is called {\it $\varphi$-singular\/},
if $\varphi_p : T_pM^n \rightarrow \mathcal{E}_p$ is not a bijection, 
where $\mathcal{E}_p$ is the fiber of $\mathcal{E}$ at $p$.
Let $S_{\varphi}$ be the set of $\varphi$-singular points.
On $M^n\setminus S_{\varphi}$,
the Levi-Civita connection of the pull-back metric $g:=\varphi^*\inner{~}{~}$
coincides with the pull-back of $D$.
Thus, we may recognize
the concept of coherent tangent bundles 
as a generalization of Riemannian manifolds. 

A coherent tangent bundle $\mathcal{E}=(M^n,\,\mathcal{E},\,\inner{~}{~},\,D,\,\varphi)$
is called {\it orientable}, if $M^n$ is orientable.
Moreover, $\mathcal{E}$ is {\it co-orientable}
if the vector bundle $\mathcal{E}$ is orientable
(i.e., there exists a smooth section $\mu$ of $\det(\mathcal{E}^{\ast})$
defined on $M^n$ such that $\mu(\vect{e}_1,\cdots,\vect{e}_n)=\pm1$
holds for any orthonormal frame $\{\vect{e}_1,\cdots,\vect{e}_n\}$ on $\mathcal{E}$).
If the $\varphi$-singular set $S_{\varphi}$ is empty, 
the orientability is equivalent to the co-orientability.
Moreover, even if the coherent tangent bundle 
is non-co-orientable (resp.\ non-orientable),
there exists a double covering $\pi : \hat{M}^n \rightarrow M^n$
such that the pull-back of $\mathcal{E}$ by $\pi$ (resp.\ $\hat{M}^n$) 
is co-orientable (resp.\ orientable) coherent tangent bundle.

\


On the other hand,
let $g$ be a positive semi-definite metric on 
a smooth $n$-manifold $M^n$.
A point $p\in M^n$ is called {\it $g$-singular\/},
if $g$ is not positive definite at $p$.
Let $S_g$ be the set of $g$-singular points.

\begin{definition}\label{def:completeCTB}
If there exists a symmetric covariant tensor $T$ on $M^n$
with compact support such that
$g+T$ gives a complete metric on $M^n$,
$g$ is called {\it complete\/}.
We call a coherent tangent bundle {\it complete\/},
if the pull-back metric is complete.
\end{definition}

Such a completeness was used  
in the study of maximal surfaces in $\R^3_1$ \cite{Kobayashi,UY},
flat fronts in $H^3$ \cite{KUY1} and in $\R^3$ \cite{MU}.
Since $S_g$ is closed, we have the following.

\begin{lemma}\label{lem:sing_compact}
The $g$-singular set $S_g$ of
a complete positive semi-definite metric $g$ 
is compact, if $S_g$ is non-empty.
If $M^n$ is compact,
every positive semi-definite metric on $M^n$ is complete.
\end{lemma}


\subsection{Frontal, induced coherent tangent bundle}
Let $(N^{n+1},\, h)$ be a Riemannian manifold
and $T_1N^{n+1}$ the unit tangent bundle.
A smooth map $f:M^n\rightarrow N^{n+1}$
is called a {\it frontal\/}, if for each point $p\in M^n$,
there exist a neighborhood $U$ of $p$ 
and a smooth vector field $\nu : U \rightarrow T_1N^{n+1}$ along $f$
such that $h(df_q(\vect{v}),\nu(q))=0$ 
holds for each $q\in U$ and $\vect{v} \in T_qM^n$.
As in \cite[Example 2.4]{SUY1},
frontals induce coherent tangent bundles.

\begin{example}[{\cite[Example 2.4]{SUY1}}]
\label{ex:induced_CTB}\,
For a frontal $f : M^n\rightarrow (N^{n+1},h)$,
set
\begin{itemize}
\item
 $\mathcal{E}_f$ is the subbundle of the pull-back bundle $f^*TN^{n+1}$ 
 perpendicular to $\nu$,
\item
 $\varphi_f : TM^n \rightarrow \mathcal{E}_f$ 
 defined as $\varphi_f(X):=df(X)$,
\item
 $D$ is the tangential part 
 of the Levi-Civita connection on $N^{n+1}$,
\item
 $\inner{~}{~}$ is the metric on $\mathcal{E}_f$ induced from 
 the metric on $N^{n+1}$.
\end{itemize}
Then, $\mathcal{E}_f=(M^n,\mathcal{E}_f,\inner{~}{~},D,\varphi_f)$
is a coherent tangent bundle.
\end{example}

We call $\mathcal{E}_f$ the {\it induced coherent tangent bundle}.
A $\varphi$-singular point of 
the induced coherent tangent bundle $\mathcal{E}_f$
is a {\it singular} point of $f$
(i.e., ${\rm rank}(df)_p<n$ holds).
A non-singular point is called {\it regular}.
A frontal $f$ is called complete 
(resp.\ orientable, co-orientable)
if the induced coherent tangent bundle $\mathcal{E}_f$ is complete 
(resp.\ orientable, co-orientable).
We remark that a frontal $f$ is co-orientable
if and only if $\nu$ is well-defined on $M^n$.

We now consider the case that 
$N^{n+1}$ is the simply connected space form 
$\Sigma^{n+1}(c)$ of constant sectional curvature $c$.
If $c=0$ (resp.\ $1$, $-1$),
the unit normal vector field $\nu$
can be considered as a smooth map into 
$S^n$ (resp.\ $S^{n+1}$, $S^{n+1}_1$).
The map $L=(f,\nu) : U \rightarrow T_1\Sigma^{n+1}(c)$
is called the {\it Legendrian lift\/} of 
$f$ (cf.\ \eqref{eq:UTbund}).

\begin{fact}[{\cite[Proposition 2.4]{SUY2}}]
Let $\Sigma^{n+1}(c)$ be the simply connected space form 
of constant sectional curvature $c$,
$f:M^n\rightarrow \Sigma^{n+1}(c)$ a frontal
with a $($locally defined$)$ unit normal vector field $\nu$, 
and $(M^n,\mathcal{E}_f,\inner{~}{~},D,\varphi_f)$ 
the induced coherent tangent bundle.
Then, the identity
 \begin{multline}\label{eq:Gauss}
  \inner{R^D(X,Y)\xi}{\zeta}\\=
  c\,\det\begin{pmatrix}
    \inner{\varphi_f(Y)}{\xi}& \inner{\varphi_f(Y)}{\zeta} \\
    \inner{\varphi_f(X)}{\xi}& \inner{\varphi_f(X)}{\zeta} 
  \end{pmatrix}
  +
  \det\begin{pmatrix}
    \inner{d\nu(Y)}{\xi}& \inner{d\nu(Y)}{\zeta} \\
    \inner{d\nu(X)}{\xi}& \inner{d\nu(X)}{\zeta} 
  \end{pmatrix}
 \end{multline}
holds,
where $X$ and $Y$ are smooth vector fields on $M^n$,
$\xi$ and $\zeta$ are smooth sections of $\mathcal{E}_f$,
and $R^D$ is the curvature tensor of the connection $D$ given by
\[
  R^D(X,Y)\xi:=D_XD_Y\xi-D_YD_X\xi-D_{[X,Y]}\xi.
\]
\end{fact}

The equation \eqref{eq:Gauss} is called the {\em Gauss equation}.

\subsection{Front}
A frontal $f : M^n \rightarrow \Sigma^{n+1}(c)$ is called a ({\it wave\/}) {\it front},
if its Legendrian lift $L$ is an immersion.
Then, $ds^2_\#:=\inner{df}{df}+\inner{d\nu}{d\nu}$
defines a Riemannian metric on $M^n$ which we call the {\it lift metric}.
If $ds^2_\#$ is complete,  $f$ is called {\it weakly complete}.
Then the following holds.

\begin{fact}[{\cite[Lemma 4.1]{MU}}]
\label{fact:MU_complete}
A complete front is weakly complete.
\end{fact}

Since the proof of \cite[Lemma 4.1]{MU}
can be applied in our setting, we omit the proof.

A point $p\in M^n$ is called {\it umbilic},
if there are real numbers $\delta_1$, $\delta_2$ 
such that $(\delta_1)^2+(\delta_2)^2\neq0$ and
$\delta_1(df)_p=\delta_2(d\nu)_p$ hold.
For a front $f : M^n \rightarrow \Sigma^{n+1}(c)$
and $\delta>0$, set
\begin{equation}\label{eq:parallel}
  f^{\delta}:=
  \left\{\begin{array}{ll}
  \vspace{1mm}
     f+\delta \nu &(\text{if}~ c=0)\\
  \vspace{1mm}
     (\cos \delta) f+(\sin \delta) \nu &(\text{if}~ c=1)\\
     (\cosh \delta) f+(\sinh \delta) \nu\qquad &(\text{if}~ c=-1)
  \end{array}\right. .
\end{equation}
Then we can check that 
$f^{\delta}$ is a front 
(called the {\it parallel front} of $f$)
with a unit normal $\nu^{\delta}$
given by
\begin{equation*}
  \nu^{\delta}:=
  \left\{\begin{array}{ll}
  \vspace{1mm}
     \nu &(\text{if}~ c=0)\\
  \vspace{1mm}
     (-\sin \delta) f+(\cos \delta) \nu &(\text{if}~ c=1)\\
     (\sinh \delta) f+(\cosh \delta) \nu\qquad &(\text{if}~ c=-1)
  \end{array}\right. .
\end{equation*}
Umbilic points are common in its parallel family.

\begin{lemma}\label{lem:singular-umb}
A singular point $p$ of a front $f$ is umbilic
if and only if $(df)_p=0$.
\end{lemma}

\begin{proof}
If $(df)_p=0$, then $p$ is umbilic.
Thus, we shall prove the converse.
Assume that $p$ is umbilic.
There exist real numbers $\delta_1$, $\delta_2$ 
$((\delta_1)^2+(\delta_2)^2\neq0)$ such that 
$\delta_1(df)_p=\delta_2(d\nu)_p$.
If we suppose $(df)_p\neq0$, then 
we have $\delta_2\neq0$,
and hence $(d\nu)_p=\delta(df)_p$,
where we put $\delta=\delta_1/\delta_2$.
Since $p$ is a singular point of $f$,
then there exists a non-zero tangent vector $\vect{v}\in T_pM^n$ 
such that $(df)_p(\vect{v})=0$.
Therefore, we have
\[
  (dL)_p(\vect{v})=((df)_p(\vect{v}),(d\nu)_p(\vect{v}))=(0,0),
\]
which contradicts the fact that $f$ is a front.
\end{proof}

\section{Coherent tangent bundle of space form type}
\label{sec:CTB}

A coherent tangent bundle 
$(M^n,\,\mathcal{E},\,\inner{~}{~},\,D,\,\varphi)$
over a smooth $n$-manifold $M^n$ 
is called {\it of constant sectional curvature} $k$, if 
\begin{equation}\label{eq:constant_curvature}
  R^D(X,Y)\xi
  =k\left( \inner{\varphi(X)}{\xi}\varphi(Y)-\inner{\varphi(Y)}{\xi}\varphi(X) \right)
\end{equation}
holds for smooth vector fields $X$, $Y$ on $M^n$
and a smooth section $\xi$ of $\mathcal{E}$.
In particular, if $k=0$, such a coherent tangent bundle is called {\it flat}.
We say a frontal $f$ has {\it constant sectional curvature} or {\it constant curvature} 
(resp.\ is {\it flat}),
if the induced coherent tangent bundle $\mathcal{E}_f$ 
has constant sectional curvature (resp.\ is flat).

\begin{definition}\label{def:CTB_spaceform}
A complete coherent tangent bundle of constant sectional curvature 
is called {\it space form type}.
\end{definition}

If the $\varphi$-singular set $S_{\varphi}$ is empty,
a coherent tangent bundle of space form type is regarded as a space form 
(i.e., a complete Riemannian manifold of constant sectional curvature).
A simply connected space form
is isometric to either a sphere, Euclidean space or hyperbolic space.
In particular, any simply connected compact space form
has positive sectional curvature.
However, if the $\varphi$-singular set $S_{\varphi}$ is non-empty,
there are simply connected compact coherent tangent bundles of space form type
whose sectional curvature is non-positive.

\begin{example}
\label{ex:simp_conn_CTB}
Set 
\begin{equation*}
  f_E : S^n \ni \vect{x}=(x_1,\cdots,x_{n},x_{n+1})
  \longmapsto (x_1,\cdots,x_{n},0) \in \R^{n+1},
\end{equation*}
where we regard $S^n$ as a subset of $\R^{n+1}$
(cf.\ Figure \ref{fig:flat_CTB}).
The singular point set of $f_E$ is the equator $\{\vect{x}\in S^n\,;\, x_{n+1}=0\}$.
Since an unit normal vector field $\nu$ along $f_E$ 
is given by $\nu=(0,\cdots,0,1)$,
$f_E$ is a co-orientable frontal (but not a front).
As in Example \ref{ex:induced_CTB},
$f_E$ induces a coherent tangent bundle
$\mathcal{E}_{f_E}=(S^n,\mathcal{E}_{f_E},\inner{~}{~},D,\varphi_{f_E})$.
Since $S^n$ is compact, 
the pull back metric $g=\varphi_{f_E}^{\ast}\inner{~}{~}$ is complete.
Moreover, the Gauss equation \eqref{eq:Gauss} implies that $R^D=0$ 
(cf.\ Lemma \ref{lem:degenerate-nu}).
Hence $\mathcal{E}_{f_E}$ is a flat coherent tangent bundle of space form type.

Similarly,
let $f_{H} : S^n \rightarrow H^{n+1}$ be the co-orientable frontal defined by
\[
  f_{H}(\vect{x})= \Phi_H \left(\frac{1}{2}(x_1,\cdots,x_{n},0)\right)
  \qquad
  \left(\vect{x}=(x_1,\cdots,x_{n},x_{n+1}) \in S^n\right),
\]
where $\Phi_H: B^{n+1} \rightarrow H^{n+1}$ is the diffeomorphism 
\[
  \Phi_H(\vect{x})=\frac{1}{1-|\vect{x}|^2}(1+|\vect{x}|^2,\,2\vect{x})
  \qquad
  \left(\vect{x}\in B^{n+1}:=\{\vect{x}\in \R^{n+1} \,;\, |\vect{x}|^2<1 \} \right).
\]
As in the case of $\mathcal{E}_{f_E}$,
the coherent tangent bundle $\mathcal{E}_{f_H}$ induced from $f_{H}$ 
is of space form type with sectional curvature $-1$.

\end{example}
\begin{figure}[htb]
\begin{center}
 \begin{tabular}{{c@{\hspace{10mm}}c@{\hspace{10mm}}c}}
  \resizebox{3cm}{!}{\includegraphics{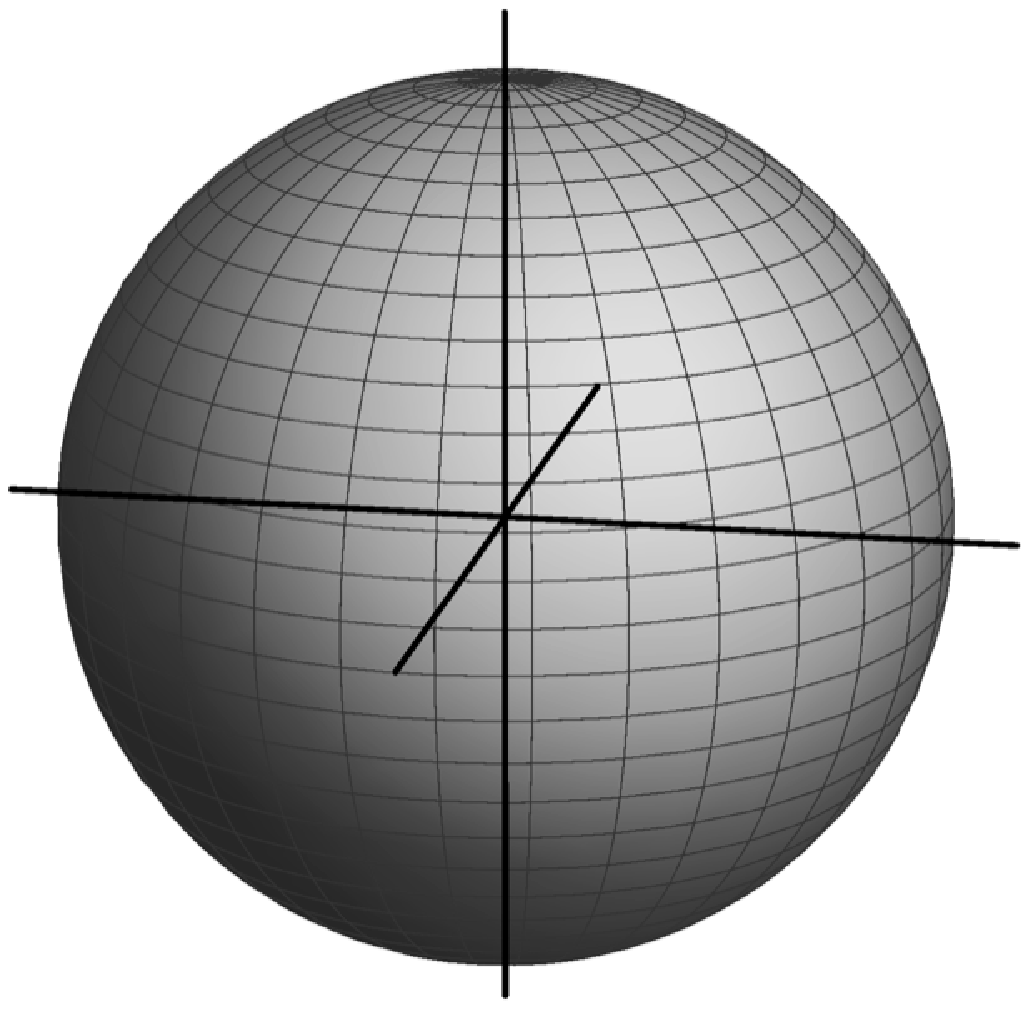}}&
  \resizebox{0.8cm}{!}{\includegraphics{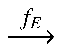}} &
  \resizebox{3cm}{!}{\includegraphics{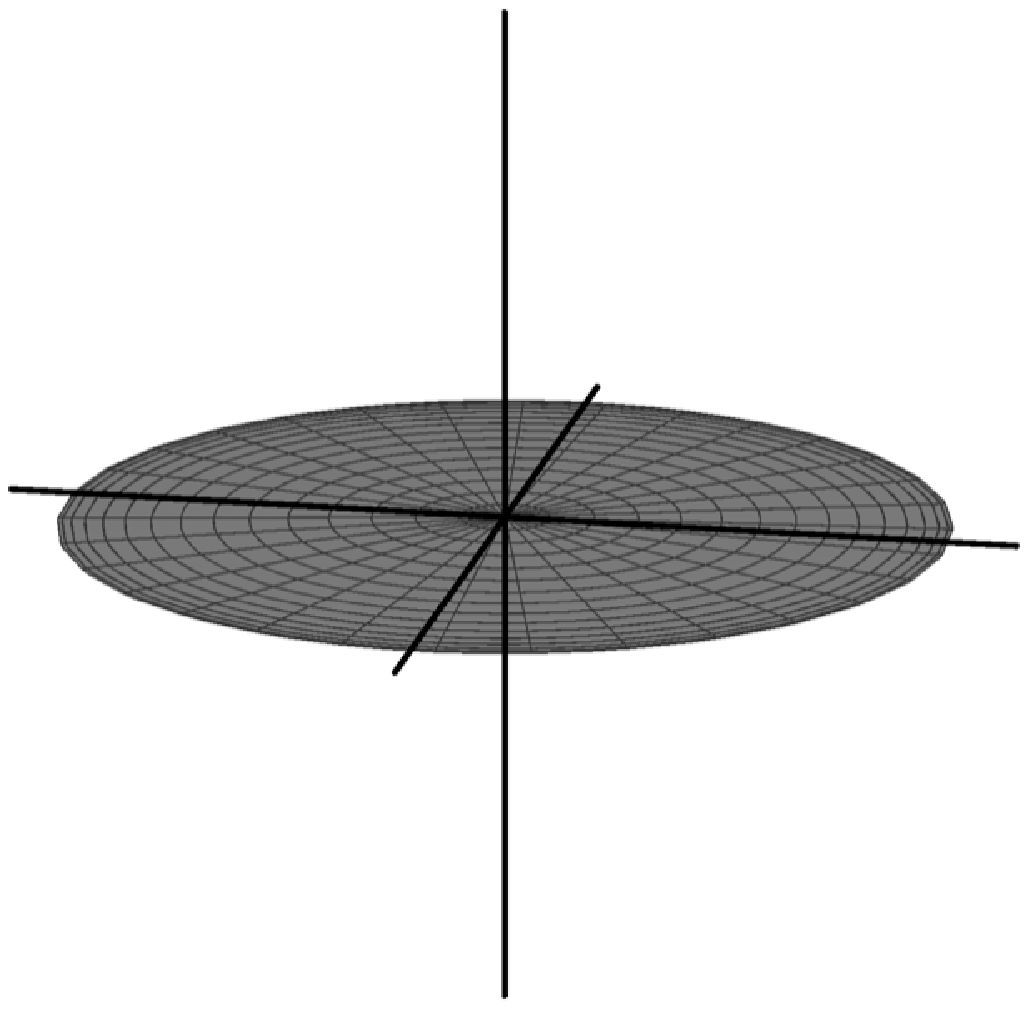}} 
 \end{tabular}
 \caption{The image of the frontal $f_E : S^n \rightarrow \R^{n+1}$
 (cf.\ Example \ref{ex:simp_conn_CTB}).
 The induced coherent tangent bundle $\mathcal{E}_{f_E}$ on $S^n$
 is flat.}
 \label{fig:flat_CTB}
\end{center}
\end{figure}


The unit normal vector field $\nu$ of
either the frontals $f_E$ or $f_H$ in Example \ref{ex:simp_conn_CTB}
satisfies ${\rm rank}(d\nu)=0$.
In general, the following holds.

\begin{lemma}\label{lem:degenerate-nu}
Let $f : M^n\rightarrow \Sigma^{n+1}(c)$ 
be a frontal with a unit normal vector field $\nu$.
Then $f$ has constant sectional curvature $c$
if and only if ${\rm rank}(d\nu)\leq1$ holds on $M^n$.
\end{lemma}

\begin{proof}
Assume $f$ is constant sectional curvature $c$,
that is, the induced coherent tangent bundle 
$\mathcal{E}_f=(M^n,\mathcal{E}_f,\inner{~}{~},D,\varphi_f)$ 
is of constant sectional curvature $c$.
By \eqref{eq:constant_curvature} and \eqref{eq:Gauss},
\begin{equation}\label{eq:degenerate-nu}
  \det\begin{pmatrix}
    \inner{d\nu(Y)}{\xi}& \inner{d\nu(Y)}{\zeta} \\
    \inner{d\nu(X)}{\xi}& \inner{d\nu(X)}{\zeta} 
  \end{pmatrix}=0
\end{equation}
holds for any vector fields $X$, $Y$ on $M^n$
and sections $\xi$, $\zeta$ of $\mathcal{E}_f$,
which is equivalent to ${\rm rank}(d\nu)\leq1$.
Conversely, assume ${\rm rank}(d\nu)\leq1$,
that is \eqref{eq:degenerate-nu} holds.
By \eqref{eq:Gauss}, 
\[
  \inner{R^D(X,Y)\xi
  -c \Bigl(\inner{\varphi_f(Y)}{\xi}\varphi_f(X)
  -\inner{\varphi_f(X)}{\xi}\varphi_f(Y)\Bigr)}{\zeta}=0
\]
holds for any vector fields $X$, $Y$ on $M^n$
and sections $\xi$, $\zeta$ of $\mathcal{E}_f$,
which implies \eqref{eq:constant_curvature} with $k=c$.
\end{proof}

\section{Classification of constant curvature $1$ fronts}
\label{sec:proof_A}

In this section, we exhibit examples of 
constant curvature $1$ fronts in the unit sphere $S^{n+1}$.
Then, we give the proof of Theorem \ref{thm:main}.
Finally, we investigate the caustics 
of constant curvature $1$ fronts and prove Theorem \ref{thm:caustic}.
In the following, we respectively denote by $f_{u_j}$, $\gamma'$
the derivatives $\partial f/\partial u_j$, $d\gamma/ds$, etc.

\subsection{Developable tubes}

The totally geodesic hypersphere $S^{n+1}\cap\{x_{n+2}=0\}$ in $S^{n+1}$
is a constant curvature $1$ front 
(immersion, in fact) satisfying ${\rm rank}(d\nu)=0$,
which can be considered as a trivial example.
We shall give an example of ${\rm rank}(d\nu)=1$ as follows.

Let $\gamma=\gamma(s) : \R\rightarrow S^{n+1}$
be a complete regular curve parametrized by arclength.
If there exists $L>0$ such that 
$\gamma(s+L)=\gamma(s)$ holds for any $s\in \R$,
we call $\gamma=\gamma(s)$ {\it closed} or {\it periodic}.
The minimum of 
$\{L>0\,;\, \gamma(s+L)=\gamma(s) \text{ for any } s\in \R \}$
is called the {\it period} of $\gamma$.
Moreover, if $\gamma(s+L/2)=-\gamma(s)$ for each $s\in \R$,
we call $\gamma(s)$ {\it antiperiodic}.
Then, $L/2$ is called the {\it antiperiod}.

\begin{definition}
For a complete regular curve $\gamma=\gamma(s)$,
set $f : \R\times S^{n-1}\rightarrow S^{n+1}$ as
\begin{equation}\label{eq:dev-tube}
  f(s,\vect{x})=x_1\vect{e}_1(s)+\cdots+x_n\vect{e}_n(s)
  \qquad
  (s\in \R,\, \vect{x}=(x_1,\cdots,x_n)\in S^{n-1}),
\end{equation}
where $\{\vect{e}_1,\cdots,\vect{e}_n\}$
is an orthonormal frame of the normal bundle $\cup_{s\in \R}(\gamma'(s))^{\perp}$.
We call $f$ the {\it tube of radius $\pi/2$} or the {\it developable tube}.
\end{definition}

If $\gamma$ is closed with period $L$ (resp.\ antiperiodic with antiperiod $L/2$),
we can regard $f$ as a map
$\bar{f} : S^1(L) \times S^{n-1}\rightarrow S^{n+1}$
(resp.\ $\bar{f} : S^1(L/2) \times S^{n-1}\rightarrow S^{n+1}$), 
where 
$S^1(L):=\R/L\Z$.
We remark that $f$ gives a parametrization of 
the set $T(\gamma)$ 
which consists of points whose geodesic distance to $\gamma=\gamma(s)$ is $\pi/2$.
Since $T(\gamma)=T(-\gamma)$,
we call $\gamma$ and $-\gamma$ the center curves of $T(\gamma)$.

\begin{lemma}\label{lem:dev-tube}
Let $f : \R\times S^{n-1}\rightarrow S^{n+1}$
be a developable tube given by \eqref{eq:dev-tube}.
Then, for each $s\in \R$, 
there exists a point $\vect{x}\in S^{n-1}$
such that $p=(s,\vect{x})\in \R\times S^{n-1}$ is a singular point of $f$.
\end{lemma}

\begin{proof}
Set $\vect{e}:=\gamma'$.
Let $\mathcal{F}=\{\gamma, \vect{e}, \vect{e}_1, \vect{e}_2,\cdots, \vect{e}_n\}$
be a {\it Bishop frame} \cite{Bishop} of $\gamma$.
That is, there exist smooth functions $\mu_j=\mu_j(s)$\, $(j=1,\cdots, n)$ such that
\[
  \vect{e}'= -\gamma +\mu_1 \vect{e}_1 +\cdots + \mu_n \vect{e}_n,\qquad
  \vect{e}_j'=-\mu_j \vect{e}
\]
for $j=1,\cdots, n$.
Remark that the curvature function $\kappa_{\gamma}=\kappa_{\gamma}(s)$
of $\gamma(s)$ is given by 
\begin{equation}\label{eq:curvature_gamma}
  \kappa_{\gamma}=\sqrt{(\mu_1)^2+\cdots+(\mu_n)^2}.
\end{equation}
Differentiating \eqref{eq:dev-tube}, we have
\begin{equation}\label{eq:differential-f}
  df=-(x_1\mu_1+\cdots+x_n\mu_n)\vect{e}\,ds
  +\vect{e}_1dx_1+\cdots \vect{e}_ndx_n.
\end{equation}
Hence,
the first fundamental form $ds^2$ is calculated as
\begin{align*}
  ds^2=\inner{df}{df}
  &=(x_1\mu_1+\cdots+x_n\mu_n)^2ds^2
       +(dx_1)^2+\cdots +(dx_n)^2\\
  &=(x_1\mu_1+\cdots+x_n\mu_n)^2ds^2+g_{S^{n-1}},
\end{align*}
where $g_{S^{n-1}}$
is the standard metric of $S^{n-1}$.
The singular point set $S_f$ is given by
\[
  S_f=\left\{ (s,\vect{x})\in \R\times S^{n-1} \,;\, x_1\mu_1(s)+\cdots+x_n\mu_n(s)=0 \right\}.
\]
At a singular point, $f_s=0$ holds.
By \eqref{eq:curvature_gamma},
we have
\[
  S_f^I:=\left\{ (s,\vect{x})\in \R\times S^{n-1} \,;\, \kappa_{\gamma}(s)=0 \right\}
\]
is a subset of $S_f$.
Hence, if $\kappa_{\gamma}(s)=0$ holds at $s\in \R$,
$(s,\vect{x})\in S_f^I$ holds for any point $\vect{x}\in S^{n-1}$.
For a singular point $(s,\vect{x})\in S_f\setminus S_f^I$,
set $\hat{\mu}(s)=(\hat\mu_1(s),\cdots,\hat\mu_n(s))$,
where $\hat\mu_j(s)=\mu_j(s)/\kappa_{\gamma}(s)$ for $j=1,\cdots, n$.
Then we have $S_f=S_f^I\cup S_f^{N\!I}$, where
\[
  S_f^{N\!I}:=\left\{ (s,\vect{x})\in \R\times S^{n-1} 
  \,;\, \inner{\hat{\mu}(s)}{\vect{x}}=0 \right\}.
\]
Thus, if $\kappa_{\gamma}(s)\neq0$ holds at $s\in \R$,
there exists $\vect{x}\in S^{n-1}$
such that $\inner{\hat{\mu}(s)}{\vect{x}}=0$,
and hence $(s,\vect{x})\in S_f^{N\!I}$.
This completes the proof.
\end{proof}

By the above proof,
it also holds that
every singular point $p\in\R\times S^{n-1}$ 
is of corank one
$($i.e., ${\rm rank}(df)_p = n-1)$.

\begin{proposition}\label{prop:dev-tube}
A developable tube 
$f : \R\times S^{n-1}\rightarrow S^{n+1}$
is an umbilic-free weakly complete constant curvature $1$ front.
If $f$ is complete, then the center curve $\gamma$ is closed.
Moreover, $f$ is non-co-orientable 
if and only if $\gamma$ is antiperiodic.
\end{proposition}

\begin{proof}
By \eqref{eq:differential-f},
we have that $\nu(s,\vect{x}):=\gamma(s)$
gives the unit normal vector field along $f$.
Thus, ${\rm rank}(d\nu)=1$ holds
and $f$ is a constant curvature $1$ frontal
by Lemma \ref{lem:degenerate-nu}.
Then, the lift metric $ds^2_{\#}$ satisfies
\[
  ds^2_{\#}=ds^2+\inner{d\nu}{d\nu}
  =\left( (x_1\mu_1+\cdots+x_n\mu_n)^2+1\right)ds^2+g_{S^{n-1}}
  \geq ds^2+g_{S^{n-1}},
\]
which implies that $f$ is a front and weakly complete.
Moreover, since ${\rm rank}(d\nu)=1$,
$f$ has no umbilic points.
Assume that $f$ is complete.
If $\gamma(s)$ is not closed,
by Lemma \ref{lem:dev-tube},
$S_f$ is not closed.
By Lemma \ref{lem:sing_compact},
the singular point set $S_f$ of $f$ is compact.
Therefore, $\gamma(s)$ must be closed.
The third assertion follows because $\gamma$
gives the unit normal of $f$.
\end{proof}

\subsection{Proof of Theorem \ref{thm:main}}
Now, we give a proof of Theorem \ref{thm:main}.
In the following lemmas
(Lemma \ref{lem:sp_coord} and Lemma \ref{lem:Massey}),
we fix a non-totally-geodesic
constant curvature $1$ front
$f:M^n\rightarrow S^{n+1}$.
First, assume $f$ is co-orientable.
Let $\nu : M^n \rightarrow S^{n+1}$ 
be a (globally defined) unit normal vector field.
By Lemma \ref{lem:degenerate-nu},
${\rm rank}(d\nu)\leq 1$.
Thus, a point $p\in M^n$ is umbilic
if and only if $(d\nu)_p=0$.
We denote by $\mathcal{U}_f$ the umbilic point set of $f$.

\begin{lemma}\label{lem:sp_coord}
A singular point $p\in M^n$ of $f$ is of corank one.
Moreover, around any non-umbilic point
$q\in M^n\setminus\mathcal{U}_f$,
there exist a local coordinate system
$(U\,;\,u_1,\cdots,u_n)$ 
and a smooth function $\rho=\rho(u_1,\cdots,u_n)$ on $U$ such that
\begin{equation}\label{eq:sp_coord}
  -\rho \nu_{u_1}=f_{u_1},\qquad
  \nu_{u_j}=0,\qquad
  \inner{\nu_{u_1}}{f_{u_j}}=0
\end{equation}
hold for each $j=2,\cdots, n$ and $\{\nu_{u_1}, f_{u_2},\cdots , f_{u_n}\}$ 
is a frame on $U$.
Setting $U_{s}=U\cap \{u_1=s\}$,
the restriction $f|_{U_{s}} : U_{s} \rightarrow S^{n-1}$ is 
a totally geodesic embedding for each $s$.
\end{lemma}

\begin{proof}
Let $(V\,;\,v_1,\cdots,v_n)$ be a local coordinate system
such that $\nu_{v_j}=0$ for each $j=2,\cdots, n$.
Since $f$ is a front,
$f_{v_j}\neq0$ for each $j=2,\cdots, n$.
Hence $p\in V$ is a singular point 
if and only if $f_{v_1}(p)=0$, and then 
$\{\nu_{v_1},\, f_{v_2},\cdots,\, f_{v_n}\}$ is linearly independent.
In this case, for each $\delta\in (0,\pi/2)$,
the parallel front $f^{\delta}$ (cf.\ \eqref{eq:parallel}) is an immersion around $p$.
Then, we have
\[
  (f^{\delta})_{v_j}=(\cos \delta)f_{v_j},\qquad
  (\nu^{\delta})_{v_j}=-(\sin \delta)f_{v_j}  
\]
for each $j=2,\cdots, n$.
Thus, the principal curvatures $\lambda_1^\delta,\cdots,\lambda_n^\delta$
of $f^{\delta}$ are given by
$\lambda_1^\delta=\lambda^{\delta}$, $\lambda_j^\delta=\tan \delta$
$(j=2,\cdots,n)$ for some function $\lambda^{\delta}$ on $V$.
Let $(U\,;\,u_1,\cdots,u_n)$ be a curvature line coordinate system of 
$f^{\delta}$ around $q\in M^n\setminus\mathcal{U}_f$.
That is,
\[
  -(\nu^{\delta})_{u_1}=\lambda^{\delta}(f^{\delta})_{u_1},\qquad
  -(\nu^{\delta})_{u_j}=(\tan \delta)(f^{\delta})_{u_j}
\]
and $\inner{(f^{\delta})_{u_1}}{(f^{\delta})_{u_j}}=0$
hold for $j=2,\cdots, n$. 
Since $\lambda^{\delta}\,(\neq \tan \delta)$ is smooth on $U$
and \eqref{eq:parallel}, we have \eqref{eq:sp_coord}
with $\rho=(\lambda^{\delta}\sin \delta+\cos \delta)/(\lambda^{\delta}\cos \delta-\sin \delta)$.
With respect to the third assertion,
$\vect{n}:=\nu_{u_1}/|\nu_{u_1}|$ gives a unit normal vector field of $f|_{U_s}$.
Then, 
\[
  \vect{n}_{u_j}= 
  \left(\frac{1}{|\nu_{u_1}|}\right)_{u_j}\nu_{u_1}+\frac{1}{|\nu_{u_1}|}\nu_{u_1u_j}
  =|\nu_{u_1}| \left(\frac{1}{|\nu_{u_1}|}\right)_{u_j} \vect{n}
\]
and $\inner{\vect{n}}{\vect{n}_{u_j}}=0$ yield
$\vect{n}_{u_j}=0$ on $U_{s}$ for each $j=2,\cdots,n$.
Together with $\nu_{u_j}=0$ $(j=2,\cdots,n)$ on $U_{s}$,
we have the conclusion.
\end{proof}

On the regular point set, the principal curvatures 
$\lambda_1,\cdots,\lambda_n$ of $f$
are given by $\lambda_1=1/\rho$, $\lambda_j=0$ $(j=2,\cdots,n)$.
We call $\rho$ and $\Lambda:=[1:\rho] : U \rightarrow P^1(\R)$ 
the {\it curvature radius function} and
the {\it principal curvature map}, respectively as in \cite{MU},
where $P^1(\R)$ is the real projective line.
We may regard $\Lambda$ as the proportional ratio 
$\Lambda=[1:\rho]=[-\nu_{u_1}:f_{u_1}]$.
Then, in our setting, 
a point $p\in M^n$ is singular (resp.\ umbilic) if and only if 
$\Lambda=[1:0]$ (resp.\ $\Lambda=[0:1]$).
By the proof above, the principal curvature map $\Lambda$ 
can be continuously extended to the whole of $M^n$ 
(cf.\ \cite[Proposition 1.6]{MU}).

Let $I$ be an open interval.
A curve $\sigma=\sigma(t) : I \rightarrow M^n$
is called {\it asymptotic},
if $d\nu(\sigma'(t))=0$ holds for each $t\in I$, where $\sigma':=d\sigma/dt$.
If $f$ is a front, $df(\sigma'(t))$ does not vanish.
Thus, we can parametrize $\sigma(t)$
by an arclength parameter $t$, 
i.e., $\inner{df(\sigma'(t))}{df(\sigma'(t))}=1$ for each $t\in I$.

\begin{lemma}\label{lem:Massey}
Let $\sigma(t) : I \rightarrow M^n$
be an asymptotic curve parametrized by arclength 
passing through a non-umbilic point 
$q=\sigma(0)\in M^n\setminus\mathcal{U}_f$.
Then the restriction of the curvature radius function $\rho$ 
to $\sigma(t)$ satisfies $\rho''(t)=\rho(t)$ for each $t\in I$.
Moreover, the closure of $\sigma(I)$ and $\mathcal{U}_f$ 
do not intersect.
\end{lemma}

\begin{proof}
Let $(U\,;\,u_1,\cdots,u_n)$
be a coordinate system centered at $q=(0,\cdots,0)$ 
as in Lemma \ref{lem:sp_coord}.
Moreover, by Lemma \ref{lem:sp_coord},
we may assume that 
each $u_n$-curve is a geodesic of 
$U_0=U\cap\{u_1=0\}$.
Lemma \ref{lem:sp_coord}
yields that $f_{u_n u_n}=f$ on $U$.
Thus, by Lemma \ref{lem:sp_coord}, 
$  (f_{u_n u_n})_{u_1}=f_{u_1}=-\rho\, \nu_{u_1}$
holds.
On the other hand, we have
$  (f_{u_1})_{u_nu_n}
  =(-\rho\, \nu_{u_1})_{u_nu_n}
  =-\rho_{u_nu_n} \nu_{u_1}.$
Since $(f_{u_n u_n})_{u_1}=(f_{u_1})_{u_nu_n}$ and $\nu_{u_1}\neq0$,
we have $\rho_{u_nu_n}=\rho$,
which proves the first assertion.
The general solution is
$
  \rho(t)=a\cos t+b\sin t
$
with constants $a$, $b\in\R$.
Therefore, we have
\begin{equation}\label{eq:Lmd1}
  \Lambda(\sigma(t))=\left[ 1 : a\cos t+b\sin t \right].
\end{equation}
Suppose that $\sigma(I)$ accumulates at an umbilic point $q\in \partial\mathcal{U}_f$.
Since $\sigma$ passes through a non-umbilic point $q$, 
there exists a sequence $\{p_m\}\subset\sigma(I)\cap(M^2\setminus \mathcal{U}_f)$
such that $\lim_{m\rightarrow \infty}p_m=q$.
Let $\{t_m\}$ be a sequence such that $p_m=\sigma(t_m)$ for each $m$.
Then, there exist constants $a,\,b$ such that $\rho(p_m)=a\cos t_m+b\sin t_m$.
In particular, $\rho$ is bounded.
On the other hand, $\Lambda=[1:\rho]=[1/\rho:1]$ 
implies that $\rho(p_m)$ must diverge as $m\rightarrow\infty$,
which is a contradiction.
Hence, we may conclude that 
the closure of $\sigma(I)$ and $\mathcal{U}_f$ have no intersection.
\end{proof}

\begin{theorem}\label{thm:w-complete}
Let $f : M^n \rightarrow S^{n+1}$ 
be a weakly complete co-orientable constant curvature $1$ front.
If $f$ is not totally geodesic, 
there exists a complete regular curve in $\gamma(s)$ in $S^{n+1}$
such that $f$ is a developable tube of $\gamma(s)$
given by \eqref{eq:dev-tube}.
\end{theorem}

\begin{proof}
Let $(U\,;\,u_1,\cdots,u_n)$
be a coordinate system centered at $q=(0,\cdots,0)$ 
as in Lemma \ref{lem:sp_coord}.
On $U$, $\nu$ is a regular curve in $S^{n+1}$
and we may represent as $\nu(u_1,\cdots,u_n)=\gamma(u_1)$.
By Lemma \ref{lem:sp_coord}, for each $s$,
there exist a $n$-subspace $\Pi_s$ of $\R^{n+2}$
such that the image of $f|_{U_{s}}$
is included in $S^{n+1}\cap \Pi_s$.
Since $\Pi_s$ is perpendicular to both of 
$\gamma(s)$ and $\gamma'(s)$,
$\Pi_s$ is the normal space $(\gamma'(s))^{\perp}$ at $\gamma(s)$.
Thus, we have the image of $f|_{U_{s}}$
is given by the embedding 
$F_s : S^{n-1} \supset \Omega \rightarrow S^{n-1}$
defined as
$F_s(\vect{x})=x_1\vect{e}_1+\cdots+x_n\vect{e}_n$,
where $\{\vect{e}_1,\cdots, \vect{e}_n\}$ is an orthonormal basis
of $(\gamma'(s))^{\perp}$ at $\gamma(s)$
and $\vect{x}=(x_1,\cdots,x_n)\in \Omega \subset S^{n-1}$,
which implies $f$ is a developable tube given by \eqref{eq:dev-tube}
of $\gamma(u_1)$ on $U$.
As in Lemma \ref{lem:dev-tube},
the lift metric $ds^2_\#$ of $f$ is given by
$ds^2_\#=(\inner{f_{u_1}}{f_{u_1}}+\inner{\nu_{u_1}}{\nu_{u_1}})(d{u_1})^2+g_{S^{n-1}}$ on $U$,
where $g_{S^{n-1}}$ is the standard metric of $S^{n-1}$.
Since $f$ is weakly complete,
we have $\Omega=S^{n-1}$.

Suppose that the umbilic point set $\mathcal{U}_f$ of $f$ is not empty,
and take an umbilic point $p\in \partial\mathcal{U}_f$.
Then there exists a sequence 
$\{q_m\}\subset M^n\setminus\mathcal{U}_f$
such that $\lim_{m\rightarrow \infty}q_m=p$.
For each $q_m$,
we set $\sigma_m$ as an asymptotic curve passing through $q_m$.
Since $\hat\sigma_m:=f\circ \sigma_m$ is a great circle,
there exists a subsequence $\{m_k\}$
such that $\hat\sigma_q=\lim_{k\rightarrow\infty}\hat\sigma_{m_k}$ 
is also a great circle.
Since the inverse image $\sigma_q$ of $\hat\sigma_q$ through $f$
passes through $q$ and by Lemma \ref{lem:Massey}, 
every point on $\sigma_q$ is umbilic.
On the other hand,
for each $\sigma_{m_k}=\sigma_{m_k}(t)$, there exist $t_k$ such that 
$\Lambda_1(\sigma_{m_k}(t_k))=[1:0]$, by \eqref{eq:Lmd1}.
If we take the limit as $k\rightarrow \infty$, we have
$\sigma_q=\lim_{k\rightarrow\infty}\sigma_{m_k}$.
Therefore, by the continuity of the principal curvature map $\Lambda$,
there exists a point on $\sigma_q$
such that $\Lambda=[1:0] \,(\neq[0:1])$,
which implies there exists a singular point on $\mathcal{U}_f$.
However, by Lemma \ref{lem:singular-umb} and Lemma \ref{lem:sp_coord},
we have a contradiction and hence $\mathcal{U}_f$ is empty.
\end{proof}

\begin{proof}[{Proof of Theorem \ref{thm:main}}]
%
By a homothety, it suffices to consider the case of $c=1$ and
we may assume $f : M^n \rightarrow S^{n+1}$ 
is a complete constant curvature $1$ front.
Assume $f$ is not totally geodesic.

First, consider the case that $f$ is co-orientable.
Since the completeness implies the weak completeness by Fact \ref{fact:MU_complete},
Theorem \ref{thm:w-complete} implies that
$f$ is a tube of radius $\pi/2$ as in \eqref{eq:dev-tube}
whose center curve $\gamma=\gamma(s) : \R \rightarrow S^{n+1}$ is regular.
By Proposition \ref{prop:dev-tube}, the center curve $\gamma$ is closed.
Let $L>0$ be the period of $\gamma$.
We shall prove the orientability of $f$.
Take a positively oriented orthonormal frame 
$\mathcal{F}=\{\gamma,\,\gamma',\,\vect{e}_1,\cdots, \vect{e}_n\}$ of $\gamma$
such that $\vect{e}=\vect{e}_1\wedge \cdots \wedge \vect{e}_n\wedge \gamma$
(cf.\ \eqref{eq:wedge_prod}). 
If $f$ is not orientable,
there exists a real number $s_0$
such that the orientation of $\mathcal{F}(s_0+L)$
does not coincide with that of $\mathcal{F}(s_0)$.
Then, 
$\vect{e}(s_0+L)
=-\vect{e}_1(s_0+L)\wedge \cdots \wedge \vect{e}_n(s_0+L)\wedge \gamma(s_0+L)$
holds.
However, by the closedness of $\gamma$, we have 
$\vect{e}_1(s_0+L)\wedge \cdots \wedge \vect{e}_n(s_0+L)\wedge \gamma(s_0+L)
=\vect{e}_1(s_0)\wedge \cdots \wedge \vect{e}_n(s_0)\wedge \gamma(s_0)$,
and hence $\vect{e}(s_0+L)=-\vect{e}(s_0)$ holds.
Then
\[
  \gamma'(s_0+L) = \vect{e}(s_0+L) = - \vect{e}(s_0) = - \gamma'(s_0)
\]
yields $\gamma'(s_0)=0$, which contradicts the regularity of $\gamma(s)$.
Therefore, $f$ must be orientable.

In the case that $f$ is not co-orientable,
there exists a double covering $\Phi : \hat{M}^n \rightarrow M^n$
such that $\hat{f}:=f\circ \Phi$ is co-orientable.
Then, $\hat{f}$ is a co-orientable developable tube
whose unit normal $\nu=\gamma$ is closed.
Moreover, since $f$ is not co-orientable,
$\nu=\gamma$ is antiperiodic and regular by Proposition \ref{prop:dev-tube}.
Assume $M^n$ is non-orientable. 
Then, by the above argument, 
we have $\nu'(s_0)=0$ for some $s_0$, which is a contradiction. 
And hence $f$ must be orientable.
\end{proof}

\subsection{Caustic}
\label{sec:caustic}

For a non-totally-geodesic constant curvature $1$ front $f:M^n\rightarrow S^{n+1}$, 
the map
\begin{equation}\label{eq:caustic}
  f^C:=\frac{1}{\sqrt{1+\rho^2}}(f+\rho\nu) 
  : M^n\setminus \mathcal{U}_f \rightarrow S^{n+1}
\end{equation}
is called the {\it caustic} (or the {\it focal hypersurface}) of $f$,
where $\rho$ is the curvature radius function (cf.\ Lemma \ref{lem:sp_coord}).
We remark that, even if $f$ is non-co-orientable,
this definition is also well-defined.
Also note that caustics may not be connected.

\begin{proposition}\label{prop:caus-tube}
For a non-totally-geodesic constant curvature $1$ front in $S^{n+1}$, 
its caustic is also a constant curvature $1$ front
without umbilic points.
Moreover, if $f$ is a developable tube as in \eqref{eq:dev-tube}
of a regular curve $\gamma=\gamma(s)$,
the unit tangent vector $\vect{e}=\gamma'/|\gamma'|$ 
gives the center curve of $f^C$. 
\end{proposition}

\begin{proof}
Using the coordinate system as in Lemma \ref{lem:sp_coord},
the first assertion can be proved by a straight-forward calculation.
If $f$ is a developable tube as in \eqref{eq:dev-tube},
we have $\rho=\mu_1(s)x_1 +\cdots+\mu_n(s) x_n$.
Substituting $f$, $\nu=\gamma$ and $\rho$ into \eqref{eq:caustic}, 
we have $\inner{f^C}{\gamma'}=\inner{(f^C)_{u_j}}{\gamma'}=0$ for $j=1,\cdots,n$,
which proves the assertion.
\end{proof}

In \cite[Theorem 6.6]{KRUY1},
Kokubu-Rossman-Umehara-Yamada proved that 
any flat front in $H^3$ is locally the caustic of some flat front.
A similar result is proved as follows.

\begin{corollary}\label{cor:inverse_caustic}
Any umbilic-free constant curvature $1$ front in $S^{n+1}$ 
is locally the caustic of some constant curvature $1$ front.
\end{corollary}

\begin{proof}
Let $f : M^n\rightarrow S^{n+1}$ be an 
umbilic-free constant curvature $1$ front.
By the proof of Theorem \ref{thm:w-complete},
$f$ is a part of developable tube.
If ${E}(s) : \R\supset I \rightarrow S^{n+1}$ gives a center curve of $f$,
let $\gamma(s): I\supset I' \rightarrow S^{n+1}$ be an integral curve of ${E}(s)$.
Then, by Proposition \ref{prop:caus-tube},
the caustic of the developable tube of $\gamma(s)$ coincides with $f$.
\end{proof}

\begin{proof}[Proof of Theorem \ref{thm:caustic}]
By Theorem \ref{thm:w-complete} and Theorem \ref{thm:main},
$f$ is a developable tube as in \eqref{eq:dev-tube}.
Let $\gamma(s): \R \rightarrow S^{n+1}$
be a center curve of $f$.
Without loss of generality,
we may assume $\gamma(s)$
is parametrized by arclength.
By Proposition \ref{prop:caus-tube},
$\vect{e}=\gamma'$ gives a center curve of the caustic $f^C$.
Let $\{ \vect{e}, \vect{e}^C:=\vect{e}'/|\vect{e}'|, \vect{e}_1^C,\cdots,\vect{e}_n^C \}$ 
be the Bishop frame of $\vect{e}$ such that 
$(\vect{e}_j^C)'=-\mu^C_j \vect{e}^C$ for $j=1,\cdots,n$.
Then the caustic is given by
$
  f^C(s,\vect{x})=x_1\vect{e}_1^C+\cdots+x_n\vect{e}_n^C
$
with the unit normal $\nu^C=\vect{e}$,
where $\vect{x}=(x_1,\cdots,x_n)\in S^{n-1}$.
If $\kappa_{\gamma}(s)$ is the curvature function of $\gamma(s)$,
$1+\kappa_{\gamma}^2=|\vect{e}'|^2$ holds.
The lift metric $ds^2_{\#,C}=\inner{df^C}{df^C}+\inner{d\nu^C}{d\nu^C}$ 
of $f^C$ is calculated as
\[
  ds^2_{\#,C}
  =\left( 1+\kappa_{\gamma}^2+ \left(\sum_{j=1}^n x_j\mu_j^C\right)^2 \right)ds^2
       +g_{S^{n-1}}
  \geq ds^2+g_{S^{n-1}},
\]
which proves the first assertion.
If $f$ is complete (resp.\ non-co-orientable),
by Proposition \ref{prop:dev-tube},
$\gamma=\gamma(s)$ is closed (resp.\ antiperiodic).
Then, $\vect{e}=\vect{e}(s)$ is also closed (resp.\ antiperiodic),
and hence the caustic $f^C$ is also complete (resp.\ non-co-orientable)
by Proposition \ref{prop:dev-tube}.
\end{proof}

\section{The case of $n=2$}
\label{sec:transformation}

In this section, 
we stick to the case of $n=2$ and investigate their duals.

\subsection{Comparison with flat fronts in $\R^3$}
In \cite{MU},
a front in $\R^3$ is called {\it flat}, 
if its (locally defined) unit normal vector field
$\nu$ degenerates everywhere.
By Lemma \ref{lem:degenerate-nu},
we have that this flatness coincides 
with our definition of flatness (cf.\ Section \ref{sec:CTB}).
Murata-Umehara \cite{MU} proved the following.

\begin{fact}[\cite{MU}]\label{fact:MU}
Let $M^2$ be a connected smooth $2$-manifold
and $f : M^2\rightarrow \R^3$ be 
a complete flat front in $\R^3$
whose singular point set is non-empty.
Then, $f$ is 
\begin{itemize}
\item umbilic-free, co-orientable and
\item orientable, more precisely $M^2$ is diffeomorphic to a circular cylinder.
\end{itemize}
Moreover, if the ends of $f$ are embedded, 
$f$ has at least four singular points 
other than cuspidal edges.
\end{fact}

Let $f : M^2\rightarrow S^3$ be a constant curvature $1$ front.
If $f$ is an immersion, the Gaussian curvature $K_{\rm int}$ 
of the induced metric is identically $1$.
By the Gauss equation $K_{\rm int}=K_{\rm ext}+1$,
the extrinsic curvature $K_{\rm ext}$ vanishes on $M^2$.
Therefore, we also call a constant curvature $1$ front
{\it extrinsically flat} (or {\it e-flat}).
By Lemma \ref{lem:degenerate-nu},
a front is e-flat if and only if 
its (locally defined) unit normal vector field
$\nu$ degenerates everywhere.
It is known that any flat immersed surface in $\R^3$
is regarded as an e-flat one in $S^3$
via a central projection (cf.\ Example \ref{ex:dev-Moebius})
and vise versa.
That is, their local properties are the same.
However, they may have different global properties.
For example, 
there are non-totally-geodesic complete flat immersions in $\R^3$ 
(cylinders over plane curves \cite{HN}),
while a complete e-flat immersion in $S^3$ must be totally geodesic \cite{OS}.

Similarly, in the case of fronts, 
the global properties of e-flat fronts in $S^3$
are not necessarily equal to those of flat fronts in $\R^3$.
For example,
although complete flat fronts in $\R^3$ with non-empty singular set
must be co-orientable by Fact \ref{fact:MU},
we can not deduce the co-orientability of complete e-flat fronts in $S^3$.
(Theorem \ref{thm:main} just implies their orientability.) 
In fact, there exist compact non-co-orientable e-flat fronts in $S^3$
as follows.

\begin{example}[Non-co-orientable e-flat front]\label{ex:non-co-ori}
Let $f: \R \times S^1 \rightarrow S^3$ be a developable tube 
as in \eqref{eq:dev-tube}.
Since $(x_1)^2+(x_2)^2=1$, we may put $x_1=\cos t$, $x_2=\sin t$.
And hence, we have the following parametrization:
\begin{equation}\label{eq:dev-tube-2}
  f(s,t)=(\cos t) \vect{e}_1(s)+(\sin t) \vect{e}_2(s).
\end{equation}
We call a regular curve in $S^3$ 
whose curvature and torsion are constant {\it helix}.
\begin{itemize}
\item
Figure \ref{fig:gallery} (B) shows 
a graphic of a developable tube $f(s,t)$ 
whose center curve is a helix $\gamma(s)$ with $(\kappa,\tau)=(3/4,5/4)$.
In this case, $f$ is co-orientable and compact.
\item
Figure \ref{fig:non-co-ori} shows 
a graphic of a developable tube $f(s,t)$ 
whose center curve is a helix $\gamma(s)$ with $(\kappa,\tau)=(4/3,5/3)$.
In this case, $\gamma(s)$ is antiperiodic, 
and hence $f$ is non-co-orientable by Proposition \ref{prop:dev-tube}.
\end{itemize}
In each case, the restriction of $f$ to the regular point set is an embedding.
By this example, we can not expect any evaluation of 
the number of singular points other than cuspidal edges 
as in the case of complete flat fronts in $\R^3$ (cf.\ Fact \cite{MU}).
\end{example}

\begin{figure}[htb]
\begin{center}
 \begin{tabular}{{c@{\hspace{8mm}}c@{\hspace{8mm}}c}}
  \resizebox{3cm}{!}{\includegraphics{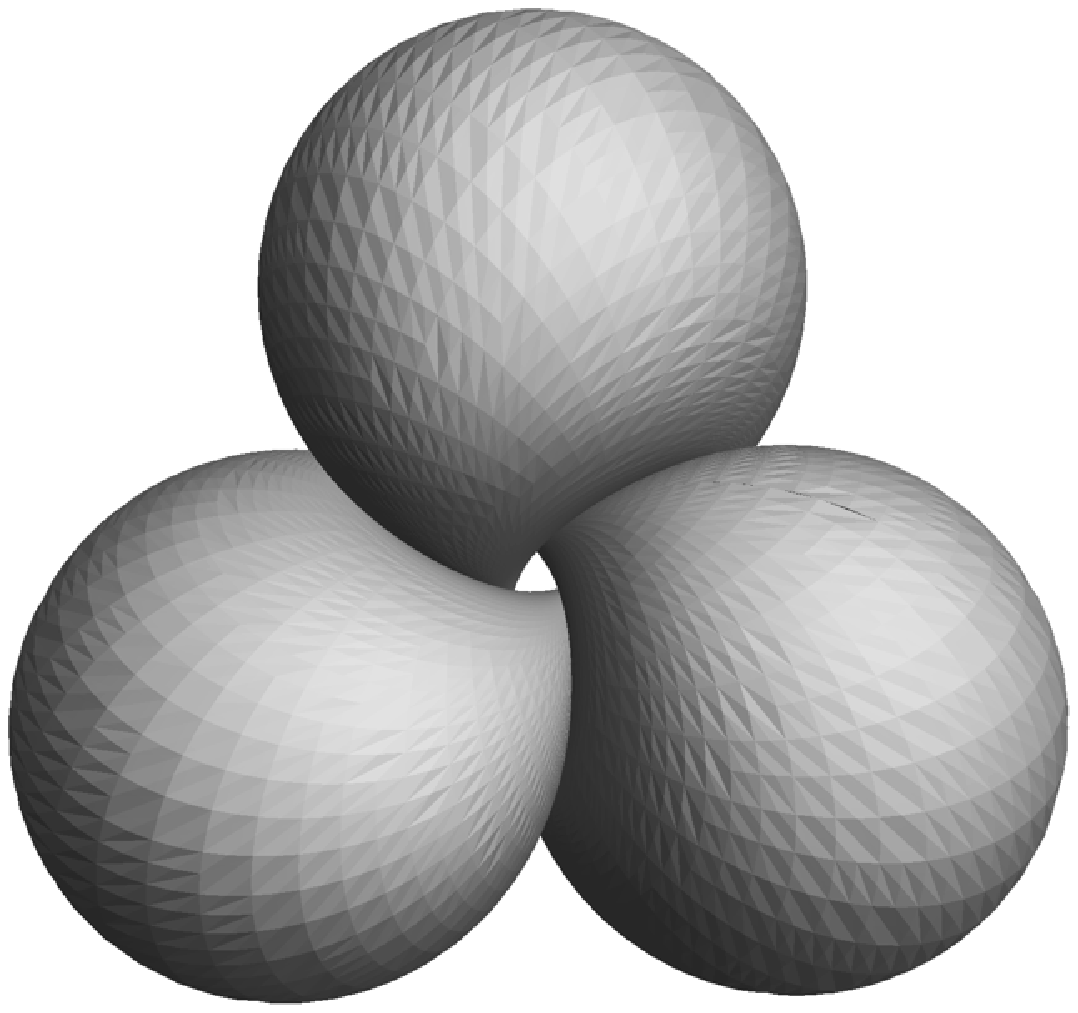}}&
  \resizebox{3cm}{!}{\includegraphics{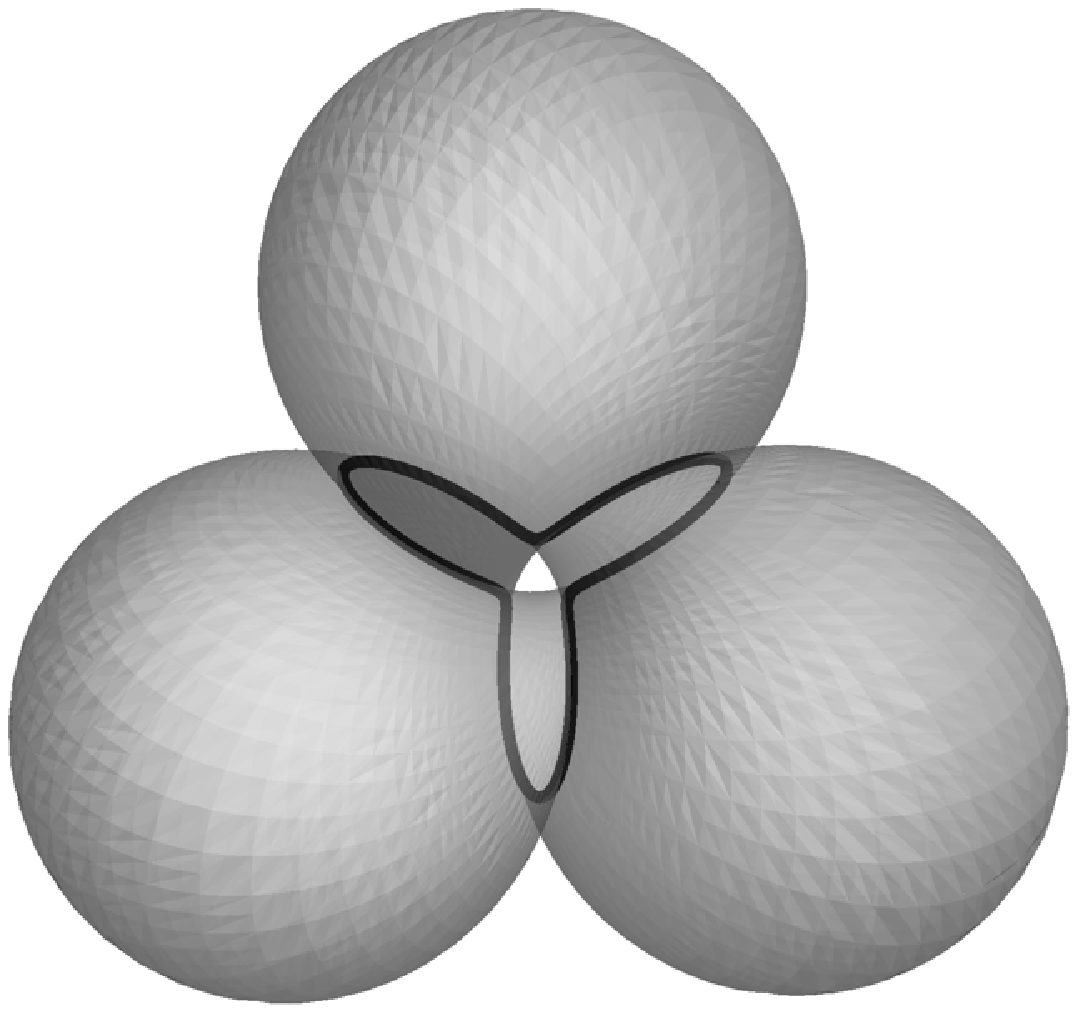}}&
  \resizebox{3cm}{!}{\includegraphics{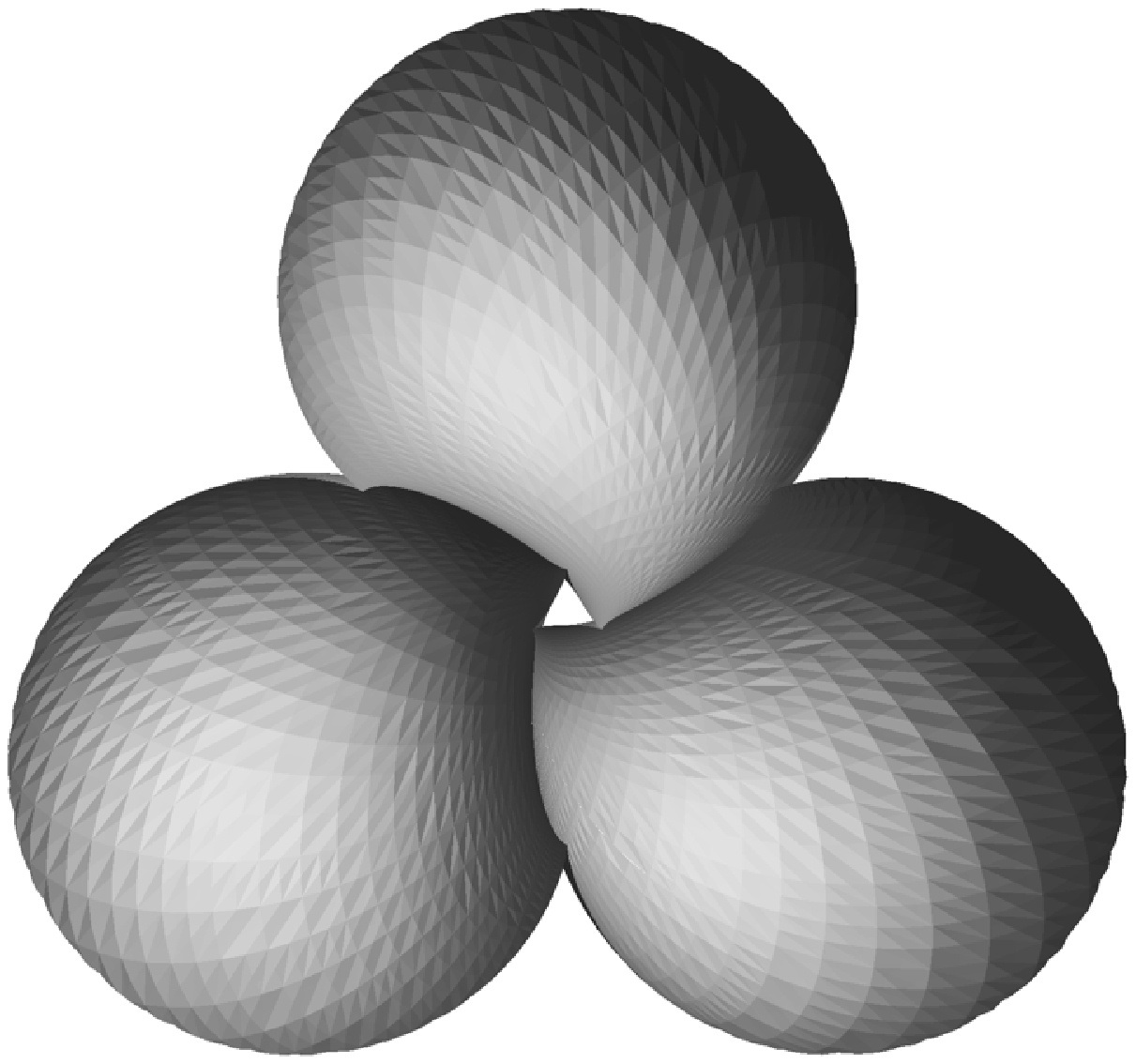}} \\
  {\footnotesize  (A) The developable tube.} &
  {\footnotesize  (B) The transparency.} &
  {\footnotesize  (C) The half cut.}
 \end{tabular}
 \caption{A non-co-orientable compact e-flat front in $S^3$,
   which is given by a developable tube whose center curve is a helix
   with $(\kappa,\tau)=(4/3,5/3)$ (cf.\ Example \ref{ex:non-co-ori}).
   The bold curve is the image of the singular set which is also a helix.
 }
 \label{fig:non-co-ori}
\end{center}
\end{figure}

It is known that for constants $a$, $b$, $\phi$ 
satisfying $a^2 \cos ^2 \phi + b^2 \sin ^2 \phi =1$,
\[
  \gamma(s)=\left(
    \cos \phi \cos as,\, \cos \phi \sin as,\, \sin \phi \cos bs,\, \sin \phi \sin bs
  \right)
\]
gives the helix with $(\kappa,\tau)=(\sqrt{(a^2-1)(1-b^2)}, ab)$
(cf.\ \cite{Tamura}).

\begin{example}[Non-orientable e-flat immersion]
\label{ex:dev-Moebius}
The central projection $\pi_{\rm cent} : S^3_+ \rightarrow \R^3 $
is a diffeomorphism defined by
\[
  \pi_{\rm cent}(x_1,x_2,x_3,x_4)=\frac{1}{x_4}(x_1,x_2,x_3),
\]
where $S^3_+$ is the hemisphere 
$S^3_+=\{ (x_1,x_2,x_3,x_4) \in S^3\,;\, x_4>0 \}$.
Then the {\it Klein model} of $S^3$ is  
$\R^3$ equipped with the metric induced by $(\pi_{\rm cent})^{-1}$.
It is known that if $f : M^2 \rightarrow \R^3$ is a flat immersion, 
$(\pi_{\rm cent})^{-1}\circ f : M^2 \rightarrow S^3$
gives an e-flat immersion.
Since there exist non-orientable flat surfaces in $\R^3$ 
(cf.\ \cite{W,KroUme,Naokawa}),
there are non-orientable e-flat ones in $S^3$ through the Klein model.
E-flat M\"obius strips are studied in \cite{Naokawa2}.
\end{example}

\subsection{Dual}
\label{sec:dual}

Here we investigate duals of e-flat fronts in $S^3$.
We first define the dual of a great circle in $S^3$.

Let $\ell = \ell(t)$ be a great circle in $S^3$.
There exists a linear $2$-subspace $\Pi_\ell$ of $\R^4$
such that $\Pi_\ell \cap S^3$ coincides with the image of $\ell$.
Then we call a great circle $\ell^\perp$ the {\it dual} of $\ell$,
if the image of $\ell^\perp$ coincides with $\Pi_\ell^{\perp} \cap S^3$,
where $\Pi_\ell^{\perp}$ is the orthogonal complement of $\Pi_\ell$ in $\R^4$.
For a great circle $\ell(t)=(\cos t)p+(\sin t)\vect{v}$ $(p\in S^3,\, \vect{v}\in T_pS^3)$,
take an orthonormal basis $\{p,\vect{v}, q, \vect{w}\}$ of $\R^4$.
Then, the dual great circle $\ell^\perp$ of $\ell$ is parametrized as
$
  \ell^\perp(t)=(\cos t)q+(\sin t)\vect{w}.
$

For an open interval $I\subset\R$,
let $f: I\times S^1 \rightarrow S^3$ be a ruled surface
$
  f(s,t)=(\cos t)p(s)+(\sin t)\vect{v}(s)
$
given by a curve $(p,\vect{v}) : I\rightarrow T_1S^3$.
As in the case of great circles, 
we may define the {\it dual} $f^\#$ as
$
  f^\#(s,t)=(\cos t)q(s)+(\sin t)\vect{w}(s),
$
where $(p,\vect{v}, q, \vect{w}) : I \rightarrow \SO(4)$ 
is an orthonormal frame.

\begin{lemma}\label{lem:dual-tube}
Let $f$ be a developable tube defined by \eqref{eq:dev-tube-2}
whose center curve is $\gamma(s)$.
Then, the binormal vector field $\vect{b}(s)$
gives the center curve of the dual $f^{\#}$.
\end{lemma}

\begin{proof}
It suffices to show that $\vect{b}(s)$ gives the unit normal of $f^{\#}$.
Since $\mathcal{F}=\{\gamma, \vect{e}=\gamma', \vect{e}_1, \vect{e}_2\}$
is an orthonormal frame, the dual $f^{\#}$ is given by
\[
  f^{\#}(s,t)=(\cos t)\gamma(s)+(\sin t)\gamma'(s).
\]
Thus 
$\inner{f^{\#}}{\vect{b}}=\inner{(f^{\#})_s}{\vect{b}}=\inner{(f^{\#})_t}{\vect{b}}=0$
hold, which proves the assertion.
\end{proof}

\begin{remark}\label{rem:INS}
In \cite[Section 7]{INS},
Izumiya-Nagai-Saji studied the dual 
of a ruled e-flat (i.e., developable) surface.
We can check that our definition of duals coincides with 
the definition given by Izumiya-Nagai-Saji \cite{INS}.
We also remark that,
since $\vect{b}'(s)=0$ at a point $\tau(s)=0$
(cf.\ \eqref{eq:Frenet-Serret}),
{\it duals of developable tube fronts are not necessarily fronts}
(see Figure \ref{fig:MU}),
where $\tau(s)$ is the torsion function of $\gamma(s)$.
\end{remark}

\begin{figure}[htb]
\begin{center}
 \begin{tabular}{{c@{\hspace{20mm}}c}}
  \resizebox{2.5cm}{!}{\includegraphics{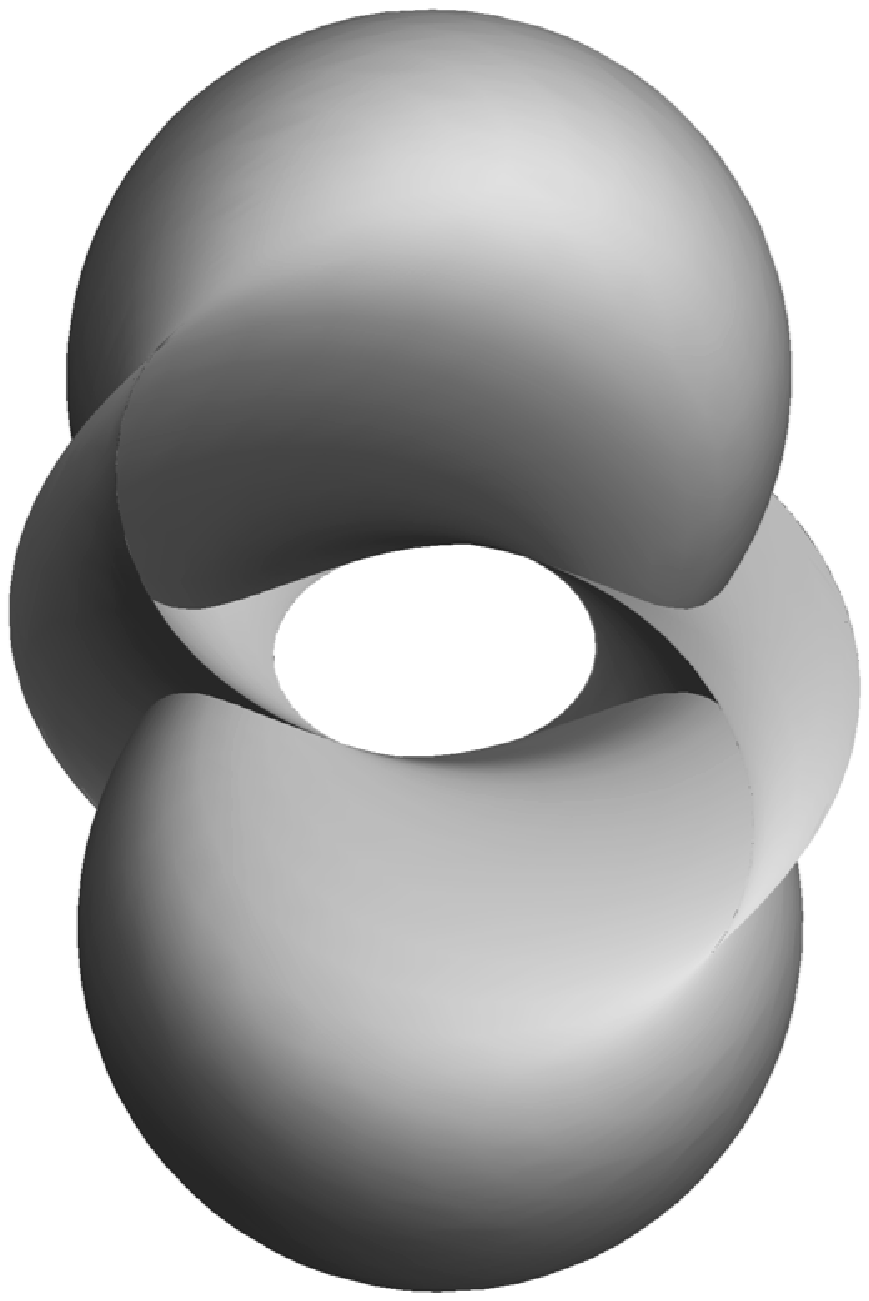}}&
  \resizebox{3.5cm}{!}{\includegraphics{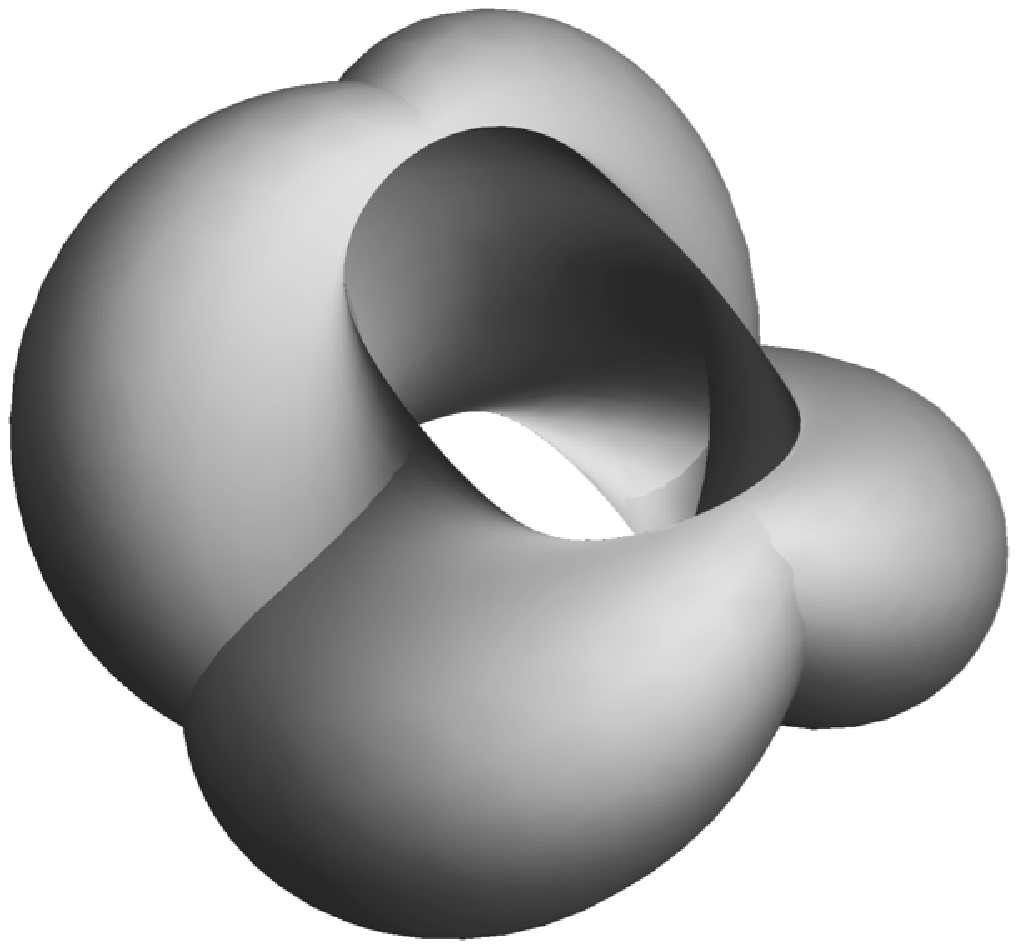}} \\
  {\footnotesize  A compact e-flat front.} &
  {\footnotesize  The dual.}
 \end{tabular}
 \caption{A compact e-flat front in $S^3$ with its dual.
 The dual has cuspidal cross cap singularities
 and hence it is not a front.
 In general, if a developable tube has swallowtails,
 the dual has cuspidal cross caps.
 Such a phenomena is called 
 {\it the duality of singularities}
 \cite[Section 7]{INS}.
 }
 \label{fig:MU}
\end{center}
\end{figure}

Now, we consider the fixed point set of the dual operation.
For a regular curve $\gamma(s)$ in $S^3$, 
the following is called the {\it Frenet-Serret} formula: 
\begin{equation}\label{eq:Frenet-Serret}
  \vect{e}'=-\gamma+\kappa\vect{n},\qquad
  \vect{n}'=-\kappa\vect{e}+\tau\vect{b}, \qquad
  \vect{b}'=-\tau\vect{n},
\end{equation}
where $\kappa=\kappa(s)$ and $\vect{e}=\vect{e}(s)$ 
are the curvature function and
the velocity vector field $\vect{e}(s)=\gamma'(s)$ 
of $\gamma$, respectively.

\begin{lemma}\label{lem:binormal-congruence}
Let $\gamma=\gamma(s) : \R\supset I\rightarrow S^3$
be a regular curve parametrized by arclength. Then
$\gamma(s)$ is congruent to $\vect{b}(s)$
if and only if $|\tau(s)|$ is identically $1$.
\end{lemma}

\begin{proof}
If $\gamma(s)$ is congruent to $\vect{b}(s)$,
$|\vect{b}'(s)|^2=|\gamma'(s)|^2=1$ holds.
On the other hand,
by the Frenet-Serret formula, we have
$|\vect{b}'(s)|^2=|\tau(s)|^2$.
Thus, we have $|\tau(s)|=1$ for each $s\in I$.
We shall prove the converse.
Without loss of generality, 
we may assume that $\tau(s)=1$ for all $s\in I$.
By the Frenet-Serret formula, we have
$\vect{b}'(s)=-\vect{n}(s)$ and 
$
  \vect{b}''(s)=-\vect{b}(s)+\kappa(s)\vect{e}(s).
$
Thus, the curvature and torsion of $\vect{b}(s)$
are given by $\kappa(s)$, $1$ respectively.
Therefore, $\gamma(s)$ is congruent to $\vect{b}(s)$.
\end{proof}

If a ruled surface is congruent to its dual,
we call it {\it self-dual}. 
By Lemma \ref{lem:dual-tube} and Lemma \ref{lem:binormal-congruence},
we have the following.

\begin{theorem}\label{prop:self-dual}
Let $f$ be a developable tube 
of a regular curve $\gamma=\gamma(s)$.
Then $f$ is self-dual if and only if $|\tau(s)|=1$.
\end{theorem}

\begin{acknowledgements}
The author would like to thank 
Professors Masaaki Umehara, Kotaro Yamada,
Miyuki Koiso, Masatoshi Kokubu, Jun-ichi Inoguchi, Yu Kawakami and Kosuke Naokawa
for their valuable comments and constant encouragements.
He also expresses gratitude to Professor Udo Hertrich-Jeromin and Gudrun Szewieczek
for careful reading of the first draft.
This work is partially supported by Grant-in-Aid for Challenging Exploratory 
Research No.\ 26610016 of the Japan Society for the Promotion of Science.
\end{acknowledgements}



\begin{thebibliography}{99}
%
\bibitem{AbeHaas}
Abe, K., Haas, A.:
Isometric immersions of $H^n$ into $H^{n+1}$.
Proc.\ Sympos.\ Pure Math.\ {\bf 54}, Part 3, 23--30, 
Amer.\ Math.\ Soc (1993).

\bibitem{AbeMoriTakahashi}
Abe, K., Mori, H., Takahashi, H.:
A parametrization of isometric immersions between hyperbolic spaces.
Geom.\ Dedicata {\bf 65}, 31--46 (1997).
  
\bibitem{Bishop}
Bishop, Richard L.:
There is more than one way to frame a curve. 
Amer. Math. Monthly {\bf 82}, 246--251 (1975).

\bibitem{Ferus}
Ferus, D.:
On isometric immersions between hyperbolic spaces. 
Math. Ann. {\bf 205}, 193--200 (1973).

\bibitem{HN}
Hartman, P., Nirenberg, L.:
On spherical image maps whose Jacobians do not change sign. 
Amer. J. Math. {\bf 81}, 901--920 (1959).

\bibitem{Honda}
Honda, A.:
Isometric immersions of the hyperbolic plane into the hyperbolic space. 
Tohoku Math. J. (2) {\bf 64}, 171--193 (2012).

\bibitem{HHNSUY}
Hasegawa, M., Honda, A., Naokawa, K., Saji, K., Umehara, M., Yamada, K.:
Intrinsic properties of surfaces with singularities. 
Internat. J. Math. {\bf 26}, 1540008, 34 pp (2015).
  
\bibitem{INS}
Izumiya, S., Nagai, T., Saji, K.: 
Great circular surfaces in the three-sphere. 
Differential Geom. Appl. {\bf 29}, 409--425 (2011).

\bibitem{KRUY1} 
Kokubu, M., Rossman, W., Umehara, M., Yamada, K.: 
Flat fronts in hyperbolic 3-space and their caustics. 
J. Math. Soc. Japan {\bf 59}, 265--299 (2007).

\bibitem{KUY1}
Kokubu, M., Umehara, M., Yamada, K.: 
An elementary proof of Small's formula for null curves in PSL(2,C) 
and an analogue for Legendrian curves in PSL(2,C). 
Osaka J. Math. {\bf 40}, 697--715 (2003).

\bibitem{KitaUme}
Kitagawa, Y., Umehara, M.: 
Extrinsic diameter of immersed flat tori in $S^3$. 
Geom. Dedicata {\bf 155}, 105--140 (2011).

\bibitem{Kobayashi}
Kobayashi, O.: 
Maximal surfaces with conelike singularities. 
J. Math. Soc. Japan {\bf 36}, 609--617 (1984). 

\bibitem{KroUme}
Kurono, Y., Umehara, M.: 
Flat {M}\"obius strips of given isotopy type in {$\mathbf{R}^3$} 
whose centerlines are geodesics or lines of curvature. 
Geom. Dedicata {\bf 134}, 109--130 (2008).

\bibitem{Massey}
Massey, William S.: 
Surfaces of Gaussian curvature zero in Euclidean 3-space. 
T\^ohoku Math. J. (2) {\bf 14}, 73--79 (1962).

\bibitem{MU}
Murata, S., Umehara, M.: 
Flat surfaces with singularities in Euclidean 3-space. 
J. Differential Geom. {\bf 82}, 279--316 (2009).

\bibitem{Naokawa}
Naokawa, K.:
Singularities of the asymptotic completion of developable {M}\"obius strips. 
Osaka J. Math. {\bf 50}, 425--437 (2013).

\bibitem{Naokawa2}
Naokawa, K.:
Extrinsically flat {M}\"obius strips on given knots in 3-dimensional spaceforms. 
Tohoku Math. J. (2) {\bf 65}, 341--356 (2013).

\bibitem{Nomizu}
Nomizu, K.: 
Isometric immersions of the hyperbolic plane into the hyperbolic space. 
Math. Ann. {\bf 205}, 181--192 (1973).

\bibitem{OS}
O'Neill, B., Stiel, E.:
Isometric immersions of constant curvature manifolds. 
Michigan Math. J. {\bf 10} 335--339 (1963).

\bibitem{Roitman}
Roitman, P.: 
Flat surfaces in hyperbolic space as normal surfaces to a congruence of geodesics. 
Tohoku Math. J. (2) {\bf 59}, 21--37 (2007). 

\bibitem{SUY2}
Saji, K., Umehara, M., Yamada, K.:
{$A_2$}-singularities of hypersurfaces with non-negative sectional curvature in Euclidean space. 
Kodai Math. J. {\bf 34}, 390--409 (2011).

\bibitem{SUY0}
Saji, K., Umehara, M., Yamada, K.:
The geometry of fronts.
Ann. of Math. (2) {\bf 169}, 491--529 (2009).

\bibitem{SUY1}
Saji, K., Umehara, M., Yamada, K.:
Coherent tangent bundles and {G}auss-{B}onnet formulas for wave fronts.
J. Geom. Anal. {\bf 22}, 383--409 (2012).

\bibitem{SUY3}
Saji, K., Umehara, M., Yamada, K.:
An index formula for a bundle homomorphism of the tangent bundle 
into a vector bundle of the same rank, and its applications.
preprint.

\bibitem{Tamura}
Tamura, M.:
Surfaces which contain helical geodesics in the 3-sphere.
Mem. Fac. Sci. Eng. Shimane Univ. Ser. B Math. Sci. {\bf 37}, 59--65 (2004).

\bibitem{UY}
Umehara, M., Yamada, K.:
Maximal surfaces with singularities in Minkowski space.
Hokkaido Math. J. {\bf 35}, 13--40 (2006).

\bibitem{W}
Wunderlich, W.: 
\"{U}ber ein abwickelbares {M}\"obiusband,
Monatsh. Math. {\bf 66}, 276--289 (1962).

\end{thebibliography}
\end{document}